\documentclass[11pt]{article}
\usepackage{amsmath,amsthm,amssymb,hyperref, color,fullpage,bm}
\usepackage{blkarray}

\newcommand{\remove}[1]{}

\makeatletter
\newtheorem*{rep@theorem}{\rep@title}
\newcommand{\newreptheorem}[2]{%
\newenvironment{rep#1}[1]{%
 \def\rep@title{#2 \ref{##1}}%
 \begin{rep@theorem}}%
 {\end{rep@theorem}}}
\makeatother

\newtheorem{thm}{Theorem}[section]
\newreptheorem{thm}{Theorem}
\newtheorem{claim}[thm]{Claim}
\newtheorem{lem}[thm]{Lemma}
\newreptheorem{lem}{Lemma}
\newtheorem{define}[thm]{Definition}
\newtheorem{cor}[thm]{Corollary}

\newtheorem{conjecture}[thm]{Conjecture}
\newreptheorem{conjecture}{Conjecture}

\newtheorem{fact}[thm]{Fact}

\hypersetup{
	colorlinks=true,
	linkcolor=blue,%
	citecolor=blue,
	linktoc=page
}
\newcommand\numberthis{\addtocounter{equation}{1}\tag{\theequation}}

\def\GL{\textsf{GL}}
\def\F{{\mathbb{F}}}
\def\E{{\mathbb{E}}}

\def\Q{{\mathbb{Q}}}
\def\Z{{\mathbb{Z}}}
\def\N{{\mathbb{N}}}

\def\C{{\mathbb{C}}}

\def\P{{\mathbb{P}}}

\def\Ind{{\mathbf{1}}}

\def\bj{{\mathbf j}}
\def\bi{{\mathbf i}}
\def\balpha{{\bm \alpha}}
\def\bzero{{\bm 0}}

\def\E{{\mathbb E}}
\def\mweight{\textsf{mweight}}

\def\rank{\textsf{rank}}

\def\EVAL{\textsf{EVAL}}

\def\Coeff{\textsf{Coeff}}

\newcommand{\eps}{\epsilon}

\usepackage{tikz-cd}

\begin{document}

\title{Maximal and $(m,\epsilon)$-Kakeya bounds over $\Z/N\Z$ for general $N$}
 \author{Manik Dhar\thanks{Department of Computer Science, Princeton University. Email: \texttt{manikd@princeton.edu}. }}

\date{}
\maketitle

\begin{abstract}
We derive Maximal Kakeya estimates for functions over $\Z/N\Z$ proving the Maximal Kakeya conjecture for $\Z/N\Z$ for general $N$ as stated by Hickman and Wright~\cite{hickman2018fourier}. The proof involves using polynomial method and linear algebra techniques from \cite{DharGeneral,arsovski2021padic,dhar2021proof} and generalizing a probabilistic method argument from \cite{dhar2022linear}. As another application we give lower bounds for the size of $(m,\eps)$-Kakeya sets over $\Z/N\Z$. Using these ideas we also give a new, simpler, and direct proof for Maximal Kakeya bounds over finite fields (which were first proven in \cite{ellenberg2010kakeya}) with almost sharp constants. 
\end{abstract}

\section{Introduction}\label{sec:intro}
Throughout this paper $[r]=\{1,\hdots,r\}$. 

We first begin with describing a special case of our results for sets as opposed to functions in general. Given a set $S\subseteq (\Z/N\Z)^n$ we can determine the largest size of intersection a line in a given direction can have with $S$. We want to lower bound the size of $S$ by using this data for every direction. We first define the set of possible directions a line can take in $(\Z/N\Z)^n$.

\begin{define}[Projective space $\P (\Z/N\Z)^{n-1}$]
For $N=p_1^{k_1}\hdots p_r^{k_r}$ where $p_1,\hdots,p_r$ are distinct primes. The Projective space $\P (\Z/N\Z)^{n-1}$ consists of vectors $u\in (\Z/N\Z)^n$ upto unit \em{($(\Z/N\Z)^{\times}$)} multiples of each other such that {\em $u$ $($mod $p_i^{k_i})$} has at least one unit co-ordinate for every $i\in [r]$.
\end{define}

For each direction in $\P (\Z/N\Z)^{n-1}$ we pick a representative in $(\Z/N\Z)^n$. This allows us to treat $\P (\Z/N\Z)^{n-1}$ as a subset of $(\Z/N\Z)^n$. We note then a line $L$ in direction $u\in \P (\Z/N\Z)^{n-1}$ is going to be a set of the form $\{a+tu|t\in \Z/N\Z\}$ for some $a\in \Z/N\Z$.  We can now state our set size lower bound.

\begin{thm}[Set size lower bounds from maximal intersection sizes over $\Z/N\Z$]\label{thm-maxSet}
Say we are given $S\subseteq (\Z/N\Z)^n$ where $N=p_1^{k_1}\hdots p_r^{k_r}$. We then have the following lower bound,
\begin{align*}
    |S|&\ge C_{N,n}\underset{u \in \P (\Z/N\Z)^{n-1}}{\E}\left[\sup\limits_{a\in (\Z/N\Z)^n}|\{a+tu|t\in \Z/N\Z\} \cap S|^n\right]\\
    &=C_{N,n}|\P (\Z/N\Z)^{n-1}|^{-1}\sum\limits_{u\in \P (\Z/N\Z)^{n-1}} \sup\limits_{a\in (\Z/N\Z)^n} |\{a+tu|t\in \Z/N\Z\} \cap S|^n,
\end{align*}
where 
$$C_{N,n}=\left(\prod_{j=1}^{r-1}\frac{1}{k_j\log(p_j)+1}\prod\limits_{i=1}^r\frac{1}{2(k_i+\lceil\log_{p_i}(n)\rceil)}\right)^n.$$
\end{thm}

Note, $\sup_{a\in (\Z/N\Z)^n} |\{a+tu|t\in \Z/N\Z\} \cap S|$ is the maximal intersection a line in direction $u$ can have with $S$. The above bound gives us new bounds for the Kakeya problem over $\Z/N\Z$ and generalizes them. To see that let us first define Kakeya sets.

\begin{define}[$(m,\eps)$-Kakeya sets]
Let $N,n\in \N$. A subset $S\subseteq (\Z/N\Z)^n$ is said to be $(m,\eps)$-Kakeya when for at least an $\eps$ fraction of directions $u\in \P (\Z/N\Z)^{n-1}$ there exists a line $L_u=\{a+tu|t\in \Z/N\Z\}$ in direction $u$ such that $|L_u\cap S|\ge m$.
\end{define}

We note $(N,1)$-Kakeya sets are often just referred to as Kakeya sets. The following is an immediate corollary of Theorem~\ref{thm-maxSet}.

\begin{cor}[$(m,\eps)$-Kakeya bounds over $\Z/N\Z$ for general $N$]\label{cor-setbds}
Say we are given an $(m,\eps)$-Kakeya set $S\subseteq (\Z/N\Z)^n$ where $N=p_1^{k_1}\hdots p_r^{k_r}$ and $p_1>\hdots>p_r$. We then have the following lower bound.
$$|S|\ge  C_{N,n}\eps  m^n,$$
where
$$C_{N,n}=\left(\prod_{j=1}^{r-1}\frac{1}{k_j\log(p_j)+1}\prod\limits_{i=1}^r\frac{1}{2(k_i+\lceil\log_{p_i}(n)\rceil)}\right)^n.$$
\end{cor}

To our knowledge, Corollary~\ref{cor-setbds} is the first $(m,\eps)$-Kakeya set bounds over $\Z/N\Z$ for general $N$ (even the case of $N$ square-free was not known). Earlier, only bounds for $(N,1)$-Kakeya sets were known for composite $N$ which were not prime powers.

Wolff in \cite{wolf1999} first asked the question of lower bounding the sizes of Kakeya sets over finite fields whose resolution might help in proving the Euclidean Kakeya conjecture. Wolff's conjecture was proven by Dvir in \cite{dvir2009size} by proving a lower bound of $q^n/n!$ for Kakeya sets in $\F_q^n$. Using the method of multiplicities and its extensions \cite{saraf2008,DKSS13,bukh2021sharp} the constant was improved from $n!^{-1}$ to $2^{-n+1}$ which is known to be tight. 

Ellenberg, Oberlin and Tao in \cite{ellenberg2010kakeya} proposed studying the size of Kakeya sets over the rings $\Z/p^k\Z$ and $\F_q[x]/\langle x^k\rangle$. They were motivated by the fact that these rings have ``scales" and hence are closer to the Euclidean version of the problem. Hickman and Wright posed the Kakeya set problem for $\Z/N\Z$ with arbitrary $N$ and considered connections between the problem and the Kakeya conjecture over the $p$-adics in \cite{hickman2018fourier}. Indeed, proving strong enough lower bounds for the sizes of Kakeya sets over the rings $\Z/p^k\Z$ and $\F_q[x]/\langle x^k\rangle$ will resolve the Minkowski dimension Kakeya set conjecture for the $p$-adic integers and and the power series ring $\F_q[[x]]$ respectively~\cite{ellenberg2010kakeya,dummithablicsek2013,hickman2018fourier}. The Kakeya problem over these rings is interesting as similarly to the Euclidean case one can construct Kakeya sets of Haar measure $0$ for the $p$-adic integers and the power series ring $\F_q[[x]]$. The constructions and generalizations to other settings can be found in \cite{dummithablicsek2013,Fraser_2016,CML_2018__10_1_3_0,hickman2018fourier,DharGeneral}.

For $n=2$ and $N=p^k$, Dummit and Hablicsek in \cite{dummithablicsek2013} proved that the sizes of Kakeya sets are lower bounded by $p^{2k}/2k$. The author and Dvir~\cite{dhar2021proof} gave almost tight Kakeya set lower bounds for $\Z/N\Z$ with square-free $N$. Arsovski in \cite{arsovskiNew} gave the first $(m,\epsilon)$-Kakeya set lower bounds for $\Z/p^k\Z$ resolving the $p$-adic Kakeya conjecture. A quantitatively stronger bound for $(m,\eps)$-Kakeya sets over $\Z/p^k\Z$ were proven in \cite{DharGeneral} (by extending the proof in \cite{arsovski2021padic} which is an earlier version of \cite{arsovskiNew} with a different proof for only $(p^k,1)$-Kakeya sets) and shown to be almost sharp for $(p^k,1)$-Kakeya sets. \cite{DharGeneral} also proved Kakeya set lower bounds over $\Z/N\Z$ for general $N$. \cite{salvatore2022kakeya} extended the techniques in \cite{arsovskiNew} giving $(m,\eps)$-Kakeya bounds over $\F_q[x]/\langle x^k\rangle$ and resolving the Kakeya set conjecture over $\F_q[[x]]$. 

The general case of our results are Maximal Kakeya bounds over $\Z/N\Z$ for general $N$ which gives norm lower bounds for functions $f:(\Z/N\Z)^n\rightarrow \Z_{\ge 0}$ using the data of the largest `intersection' a line in a direction can have with $f$. The size of `intersections' of a line $L$ and $f$ is $\sum_{x\in L} f(x)$. This information about `intersections' is encapsulated using the following definition.

\begin{define}[Maximal function]
Given $f: (\Z/N\Z)^n\rightarrow \Z_{\ge 0}$ we define the maximal function $f^*:\P (\Z/N\Z)^{n-1}\rightarrow \Z_{\ge 0}$ as follows,
$$f^*(u)= \sup\limits_{a\in (\Z/N\Z)^n} \left(\sum\limits_{t\in \Z/N\Z} f(a+tu)\right).$$
\end{define}

The most general result we prove will lower bound the norm of $f$ by the norm of $f^*$. Before describing the statement in its full generality we first state it for $N=p^k$.

\begin{thm}[Maximal Kakeya bounds over $\Z/p^k\Z$]\label{thm-maxpk}
Let $n>0$ be an integer, $p$ prime and $k\in \N$. For any function $f:(\Z/p^k\Z)^n\rightarrow \N$ we have the following bound,
$$\sum\limits_{x\in (\Z/p^k\Z)^n} |f(x)|^n \ge C_{p^k,n}\underset{u \in \P (\Z/p^k\Z)^{n-1}}{\E}[|f^*(u)|^n]= \frac{C_{p^k,n}}{|\P(\Z/p^k\Z)^{n-1}|} \left(\sum\limits_{u\in \P (\Z/p^k\Z)^{n-1}} |f^*(u)|^n\right),$$
where
$$C_{p^k,n}= \frac{1}{(2(\lceil\log_p(\max_u f^*(u))+ \log_{p}(n)\rceil ))^n}.$$
When $p>n$ we can improve $C_{p^k,n}$ to,
$$C_{p^k,n}=(\lceil\log_p(\max_u f^*(u))\rceil+1)^{-n} (1+n/p)^{-n}.$$
\end{thm}

We note $\log_p(\max_u f^*(u))$ is bounded above by $k+\log_p(\|f\|_{\ell^\infty})$ which means that as long as the terms are not too big we can take the above inequality as a comparison between the $\ell^n$ norms of $f$ and $f^*$. 

Bounds of these types were first proven for finite fields in \cite{ellenberg2010kakeya}. Our results will prove these bounds over $\Z/N\Z$ for general $N$.

To state the general case of composite $N$ we first need some simple facts which follow from the Chinese remainder theorem.

\begin{fact}[Geometry of $\Z/p^kN_0\Z$]\label{fact:geoChine}
Let $p,N_0,n,k\in \N,R=\Z/p^kN_0\Z,R_0=\Z/N_0\Z$ with $p$ prime and co-prime to $N_0$.  Using the Chinese remainder theorem we know that any co-ordinate in $R^n$ can be uniquely represented by a tuple in $(\Z/p^k\Z)^n\times R_0^n$. Also any direction in the projective space $\P R^{n-1}$ can again be uniquely represented by a tuple in $\P (\Z/p^k\Z)^{n-1} \times \P R_0^{n-1}$. Finally, any line $L$ with direction $b=(b_p,b_0)\in \P (\Z/p^k\Z)^{n-1} \times \P R_0^{n-1}$ in $R^n$ is equivalent to the product of a line $L_p\subset (\Z/p^k\Z)^n$ in direction $b_p$ and a line $L_0\subset R_0^n$ in direction $b_0$. 
\end{fact}

We will need to define a new quantity on which our bounds will depend.

\begin{define}[$p$-Maximal weight]
For $p$ and $N$ coprime and $f:(\Z/p^kN\Z)^n\rightarrow \N$ we define the $p$-maximal weight {\em $\mweight(f,p)$} as follows:

Let $L(u)=\{a_u+tu|t\in \Z/p^kN\Z\}$ be a line such that $\sum_{x\in L(u)} f(x)=f^*(u)$. Using Fact~\ref{fact:geoChine} we note that the line $L(u)$ can be written as a product of lines $L_p(u)\subseteq (\Z/p^k\Z)^n$ and $L_1(u)\subseteq (\Z/N\Z)^n$. We define,
{\em $$\mweight(f,p)=\sup_{u\in \P (\Z/p^kN\Z)^{n-1},z\in L_1(u)} \sum\limits_{x\in L_p(u)} f((x,z)).$$}
\end{define}

Note, $\mweight(f,p)$ for $f:(\Z/p^k\Z)^n\rightarrow \N$ is simply $\max_u f^*(u)$.

We can finally state our main theorem.

\begin{thm}[Maximal Kakeya bounds over $\Z/N\Z$ for general $N$]\label{thm-maxN}
Let $n>0$ be an integer and $N=p_1^{k_1}\hdots p_r^{k_r}$ with $p_i$ primes and $k_i\in \N$. For any function $f:(\Z/N\Z)^n\rightarrow \N$ we have the following bound,
$$\sum\limits_{x\in (\Z/N\Z)^n} |f(x)|^n \ge  C_{N,n}\underset{u \in \P (\Z/N\Z)^{n-1}}{\E}[|f^*(u)|^n]=\frac{C_{N,n}}{|\P(\Z/N\Z)^{n-1}|} \left(\sum\limits_{u\in \P (\Z/N\Z)^{n-1}} |f^*(u)|^n\right),$$
where
{\em \begin{align*}
    C_{N,n}=&\left(\frac{1}{2(\log(\mweight(f,p_1))+1)\lceil \log_{p_1}(\mweight(f,p_1))+\log_{p_1}(n)\rceil}\right)^n\\
    &\cdot \left(\frac{1}{2(k_r+\lceil \log_{p_r}(n)\rceil)}  \prod\limits_{i=2}^{r-1}\frac{1}{2(k_i\log(p_i)+1)(k_i+\lceil \log_{p_i}(n)\rceil)}\right)^n
\end{align*}.}
\end{thm}

Similarly to the comment after Theorem~\ref{thm-maxpk} we note that $\log(\mweight(f,p_1))$ is bounded above by $k_1\log(p_1)+\log(\|f\|_{\ell^\infty})$. As before as long as the terms are not too big we can take the above inequality as a comparison between the $\ell^n$ norms of $f$ and $f^*$. 

The overall argument will be via induction. The base case of Theorem~\ref{thm-maxN} is Theorem~\ref{thm-maxpk} (in fact the constants are slightly better for the $r=1$ case). We also note if we apply Theorem~\ref{thm-maxN} for $f$ equaling the indicator function $\Ind_S:(\Z/N\Z)^n\rightarrow \Z_{\ge 0}$ of a set $S$ we will get Theorem~\ref{thm-maxSet} (it is easy to see that $\mweight(\Ind_S,p_1)\le p_1$). Recall, $\Ind_S(x)=1$ if $x\in S$ and $0$ otherwise.

The main ideas in the proof involve using and extending the probabilistic and the polynomial method tools from the papers \cite{dhar2021proof,arsovski2021padic,DharGeneral,dhar2022linear}. In particular, generalizing the new probabilistic method arguments from \cite{dhar2022linear} is what allows us to overcome barriers to deal with the case of composite $N$ with multiple prime factors. As mentioned earlier, even $(m,\epsilon)$-Kakeya bounds were not known for non prime power $N$ (even if $N$ was square-free). Indeed, already known techniques (in particular, combining the arguments from \cite{arsovskiNew} and \cite{ellenberg2010kakeya}) would have sufficed to prove maximal bounds for the case of prime powers.

In addition, we also show that a simple application of the arguments in \cite{dhar2022linear} leads to a direct and simpler proof for Maximal Kakeya bounds over finite fields with almost sharp constants.

\begin{thm}[Maximal Kakeya bounds over finite fields]\label{thm-field}
Let $n>0$ be an integer, $q$ a prime power. For any function $f:\F_q^n\rightarrow \C$ we have the following bound,
$$\sum\limits_{x\in \F_q^n} |f(x)|^n \ge  \frac{1}{2^n}\underset{u \in \P \F_q^{n-1}}{\E}[|f^*(u)|^n]=\frac{1}{2^n} |\P \F_q^{n-1}|^{-1} \left(\sum\limits_{u\in \P \F_q^{n-1}} |f^*(u)|^n\right).$$
\end{thm}

The first bounds of this type were proven in \cite{ellenberg2010kakeya} where it was also proven for the more general case of varities by reducing to the case of the case of $\F^n_q$. Our arguments can be combined with theirs to improve the constants but we only focus on the case of $\F^n_q$ as that is where our new argument is being used.

The sharpness of the above theorem is evident if we take the indicator function of Kakeya sets which then gives us the lower bounds from \cite{DKSS13} which are tight up to a factor of $2$.\footnote{Indeed our arguments can be combined with the arguments of \cite{bukh2021sharp} to get rid of this $2$ but we do not do that to simplify the proof.}  Finally, we note that Theorem~\ref{thm-maxN} can be used to resolve the following statement of the Kakeya maximal conjecture in \cite{hickman2018fourier}. 

\begin{conjecture}[Kakeya Maximal conjecture over $\Z/N\Z$]\label{conj}
For all $\eps > 0$ and $n\in \N$ there exists a constant $C_{n,\eps}$ such that the following holds:
For a choice of a line $L(u)$ for each direction $u\in \P (\Z/N\Z)^{n-1}$ we have,
$$ \left\|\sum\limits_{u \in \P (\Z/N\Z)^{n-1}} \Ind_{L(u)}\right\|_{\ell^{n/(n-1)}} \leq C_{n,\epsilon} N^{\eps} \left(\sum\limits_{u \in \P (\Z/N\Z)^{n-1}} |L(u)| \right)^{(n-1)/n}.$$
\end{conjecture}

The above inequality simply follows from taking the dual of Theorem~\ref{thm-maxN}. This will be shown in Section~\ref{sec-conj}.

\subsection{Acknowledgements}
The author would like to thank Zeev Dvir and Suryateja Gavva for helpful comments and discussion. Research supported by NSF grant DMS-1953807.

\subsection{Paper Organization:}
In Section~\ref{sec-pre} we state preliminaries from \cite{DKSS13,dhar2021proof,dhar2022linear} which we will need. In Section~\ref{sec-field} we prove Theorem~\ref{thm-field} with our new techniques. In Section~\ref{sec-base} we prove Theorem~\ref{thm-maxpk} which is the base case of our induction argument. In Section~\ref{sec-induct} we finally prove Theorem~\ref{thm-maxN}. In Section~\ref{sec-conj} we prove Conjecture~\ref{sec-conj}.

\section{Preliminaries}\label{sec-pre}
\subsection{Multiplicities and Hasse derivative}
We first review the definitions of multiplicities and Hasse derivatives that will be needed in the proof (see  \cite{DKSS13} for a more detailed discussion). We will allow the definitions to be over an arbitrary field $\F$ since we will need to apply them both for $\F=\F_q$ (for Theorem~\ref{thm-field}) and also for $\F = \C$ (for Theorem~\ref{thm-maxN} and Theorem~\ref{thm-maxpk}).

\begin{define}[Hasse Derivatives]
Let $\F$ be a field. Given a polynomial $Q\in \F[x_1,\hdots,x_n]$  and an $\mathbf{i}\in \Z_{\ge 0}^n$ the $\mathbf{i}$th {\em Hasse derivative} of $Q$ is the polynomial $Q^{(\mathbf{i})}$ in the expansion $$Q(x+z)=\sum_{\mathbf{j}\in \Z_{\ge 0}^n} Q^{(\mathbf{j})}(x)z^{\mathbf{j}}$$ where $x=(x_1,...,x_n)$, $z=(z_1,...,z_n)$ and $z^{\mathbf{j}}=\prod_{k=1}^n z_k^{j_k}$.  
\end{define}

Hasse derivatives satisfy the following useful property (see \cite{DKSS13} for a proof). We will only need this property  to show that, if $Q^{(\mathbf{i}+\mathbf{j})}$ vanishes at a point then so does $(Q^{(\mathbf{i})})^{(\mathbf{j})}$.

\begin{lem}\label{lem:chainRule}
Given a polynomial $Q\in \F[x_1,\hdots,x_n]$ and $\mathbf{i},\mathbf{j}\in \Z_{\ge 0}^n$, we have 
$$(Q^{(\mathbf{i})})^{(\mathbf{j})}=Q^{(\mathbf{i}+\mathbf{j})}\prod\limits_{k=1}^n\binom{i_k+j_k}{i_k}$$
\end{lem}

We make precise what it means for a polynomial to vanish on a point $a\in \F^n$ with multiplicity. First we recall for a point $\mathbf{j}$ in the non-negative lattice $\Z^n_{\ge 0}$, its weight is defined as $\text{wt}(\mathbf{j})=\sum_{i=1}^n j_i$.

\begin{define}[Multiplicity]
For a polynomial $Q\in \F[x_1,\hdots,x_n]$ and a point $a\in \F^n$ we say $Q$ vanishes on $a$ with {\em multiplicity} $m\in \Z_{\geq 0}$, if $m$ is the largest integer such that all Hasse derivatives of $Q$ of weight strictly less than $m$ vanish on $a$. We use {\em $\textsf{mult}(Q,a)$} to refer to the multiplicity of $Q$ at $a$.
\end{define}

Note that the number of Hasse derivatives over $\F[x_1,\hdots,x_n]$ with weight strictly less than $m$ is $\binom{n+m-1}{n}$. Hence, requiring that a polynomial vanishes to order $m$ at a single point $a$ enforces the same number of homogeneous linear equations on the coefficients of the polynomial.  We will use the following simple property concerning multiplicities of composition of polynomials (see \cite{DKSS13} for a proof).

\begin{lem}\label{lem:multComp}
Given a polynomial $Q\in \F[x_1,\hdots,x_n]$ and a tuple $G=(g_1,\hdots,g_n)$ of polynomials in $\F[y_1,\hdots,y_m]$, and $a\in \F^m$ we have, 
{\em $$\textsf{mult}(Q\circ G, a)\ge \textsf{mult}(Q,G(a)).$$}
\end{lem}

Finally, we need a lemma which given a $\epsilon$ fraction of directions $B\subseteq \P \F_q^{n-1}$ we can find a large set of monomials of a fixed degree (not depending on $\epsilon$) such that any linear combination of these monomials doesn't vanish on the entirety of $B$ with high multiplicity.

\begin{lem}\label{lem-goodmono}
Let $r\in \N$ and $B\subseteq \P \F_q^{n-1}$ with $|S|\ge \epsilon |\P \F_q^{n-1}|,\epsilon\in [0,1]$ then for any $d<rq$ there exists a set $P_B(d,r), |P_B(d,r)|=\epsilon \binom{d+n}{n}$ of monomials of degree exactly $d$ such that no non-zero linear combination of monomials in $P_B(d,r)$ vanishes with multiplicity at least $r$ over all points in $B$. 
\end{lem}

This is analogous to Corollary 6.12 in \cite{dhar2022linear}. The proof is nearly identical and uses a probabilistic method argument with the Schwartz-Zippel lemma~\cite{schwartz1979probabilistic,ZippelPaper}. For completeness, we give a proof in the appendix.

\subsection{Rank of matrices with polynomial entries}\label{sec:PreDD}
We will be needing a number of preliminaries from \cite{DharGeneral} which we state without proof.

\begin{define}[Rank of matrices with entries in {$\mathbb{F}[z]/ \langle f(z)\rangle$}]
Given a field $\F$ and a matrix $M$ with entries in $\F[z]/\langle f(z)\rangle$ where $f(z)$ is a non-constant polynomial in $\F[z]$, we define the $\F$-rank of $M$ denoted {\em $\rank_\F M$}  as the maximum number of $\F$-linearly independent columns of $M$.
\end{define}





Given a matrix $A$ with entries in $\F[z]/\langle f(z)\rangle$ we want to construct a new matrix with entries only in $\F$ such that their $\F$-ranks are the same. First we state a simple fact about $\F[z]/\langle f(z)\rangle$.

\begin{fact}[Unique representation of elements in $\F [z\text{]}/\langle f(z)\rangle$]
Let $\F$ be a field and $f(z)$ a non-constant polynomial in $\F[z]$ of degree $d>0$. Every element in $\F[z]/\langle f(z)\rangle$ is uniquely represented by a polynomial in $\F[z]$ with degree strictly less than $d$ and conversely every degree strictly less than $d$ polynomial in $\F[z]$ is a unique element in $\F[z]/\langle f(z)\rangle$. When we refer to an element $h(z)\in \F[z]/\langle f(z)\rangle$ we also let it refer to the unique degree strictly less than $d$ polynomial it equals.
\end{fact}

\begin{define}[Coefficient matrix of $A$]
Let $\F$ be a field and $f(z)$ a non-constant polynomial in $\F[z]$ of degree $d>0$. Given any matrix $A$ of size $n_1\times n_2$ with entries in $\F[z]/\langle f(z)\rangle$ we can construct the {\em coefficient matrix of $A$} denoted by {\em $\Coeff(A)$} with entries in $\F$ which will be of size $dn_1\times n_2$ whose rows are labelled by elements in $\{0,\hdots,d-1\}\times [n_1]$ such that its $(i,j)$'th row is formed by the coefficients of $z^i$ of the polynomial entries of the $j$'th row of $A$. 
\end{define}

The key property about the coefficient matrix immediately follows from its definition.

\begin{fact}\label{lem:coeffR}
Let $\F$ be a field and $f(z)$ a non-constant polynomial in $\F[z]$ of degree $d>0$. Given any matrix $A$ with entries in $\F[z]/\langle f(z)\rangle$ and its coefficient matrix {\em $\Coeff(A)$} it is the case that an $\F$-linear combination of a subset of columns of $A$ is $0$ if and only if the corresponding $\F$-linear combination of the same subset of columns of {\em $\Coeff(A)$} is also $0$.

In particular, the $\F$-rank of $A$ equals the $\F$-rank of {\em $\Coeff(A)$}.
\end{fact}

We now need some simple properties related to the crank of tensor products. To that end we first define the tensor/Kronecker product of matrices.

\begin{define}[Kronecker Product of two matrices]
Given a commutative ring $R$ and two matrices $M_A$ and $M_B$ of sizes $n_1\times m_1$ and $n_2\times m_2$ corresponding to $R$-linear maps $A:R^{n_1}\rightarrow R^{m_1}$ and $B:R^{n_2}\rightarrow R^{m_2}$ respectively, we define the Kronecker product $M_A\otimes M_B$ as a matrix of size $n_1n_2\times m_1m_2$ with its rows indexed by elements in $[n_1]\times [n_2]$ and its columns indexed by elements in $[m_1]\times [m_2]$ such that
$$M_A\otimes M_B ((r_1,r_2),(c_1,c_2))=M_A(r_1,c_1)M_B(r_2,c_2),$$
where $r_1\in [n_1],r_2\in [n_2],c_1\in [m_1]$ and $c_2\in [m_2]$. $M_A\otimes M_B$ corresponds to the matrix of the $R$-linear map $A\otimes B: R^{n_1}\otimes R^{n_2}\cong R^{n_1n_2}\rightarrow R^{m_1}\otimes R^{m_2}\cong R^{m_1m_2}$.
\end{define}

We will need the following simple property of Kronecker products which follows from the corresponding property of the tensor product of linear maps.

\begin{fact}[Multiplication of Kronecker products]\label{fact:multOfKronecker}
Given matrices $A_1,A_2,B_1$ and $B_2$ of sizes $a_1\times n_1$, $a_2\times n_2$, $n_1\times b_1$ and $n_2\times b_2$ we have the following identity,
$$(A_1 \otimes A_2) \cdot (B_1\otimes B_2)=(A_1\cdot B_1)\otimes (A_2\cdot B_2).$$
\end{fact}

We also need a simple fact about tensor product of a set of linearly independent rows with arbitrary vectors. 

\begin{fact}[Linear independence of Kronecker products]\label{fact-tenLin}
Let $k,n,m\in \N$. Given linearly independent vectors $r_1,r_2,\hdots,r_k\in \F^n$ then the set of vectors $r_i\otimes \F^m,i=[k]$ are mutually disjoint from each other outside of $0$. In particular, the union of a set of linearly independent vectors from each of $r_i\otimes \F^m,i=[k]$ gives us a linearly independent set of vectors.
\end{fact}




\subsection{Polynomial Method on the complex torus}\label{sec:PreAr}
This section also includes preliminaries from \cite{DharGeneral} but written for a slightly more general set of parameters. In most cases the proofs are near identical as in Section 3 of \cite{DharGeneral}. Where the proofs are more complicated we give proofs in the appendix for completeness.

We first define the rings that we will work over.
\begin{define}[Rings $\overline{T}_{\ell}$ and $T_{\ell}^k$]\footnote{In \cite{DharGeneral} $\overline{T}_{\ell}$ is defined as $\F_p[z]/\langle z^{p^\ell}-1\rangle$. We will need this more general definition in this paper.}
We let 
$$\overline{T}_{\ell}=\F_p[z]/\langle (z-1)^\ell\rangle$$ 
and 
$$T_{\ell}^k=\Z(\zeta_{p^k})[z]/\langle (z-1)^\ell\rangle$$
where $\zeta_{p^k}$ is a primitive complex $p^k$'th root of unity.
\end{define}

In \cite{DharGeneral} $\ell$ is a power of $p$ but that is not necessary for the proofs to work. To prove maximal Kakeya bounds we need this extra control over the parameters.

We suppress $p$ in the notation as it will be fixed to a single value throughout our proofs. We also let $\zeta=\zeta_{p^k}$ throughout this section for ease of notation.

Note that $\Z(\zeta)$ is the ring $\Z[\zeta]/\langle \phi_{p^k}(\zeta)\rangle$ where,
\begin{align}
    \phi_{p^k}(x)=\frac{x^{p^k}-1}{x^{p^{k-1}}-1}=\sum\limits_{i=0}^{p-1} x^{ip^{k-1}}\label{eq:cyclo},
\end{align}
is the $p^k$ Cyclotomic polynomial. We will also work with the field $\Q(\zeta)=\Q[\zeta]/\langle \phi_{p^k}(\zeta)\rangle$. 

We need a simple lemma connecting $\Z(\zeta)$ to $\F_p$.

\begin{lem}[Quotient map $\psi_{p^k}$ from $\Z(\zeta)$ to $\F_p$]
The field $\F_p$ is isomorphic to $\Z(\zeta)/\langle p,\zeta-1\rangle$. In particular, the map $\psi_{p^k}$ from $\Z(\zeta)$ to $\F_p$ which maps $\Z$ to $\F_p$ via the mod $p$ map and $\zeta$ to $1$ is a ring homomorphism.
\end{lem}

We note $\psi_{p^k}$ can be extended to the rings $\Z(\zeta)[z]/\langle h(z)\rangle$ for any $h(x)\in \Z(\zeta)[x]$ by mapping $z$ to $z$. In particular, we also have the following.

\begin{cor}[Extending $\psi_{p^k}$]\label{cor:psiE}
Let $h(x)\in \Z(\zeta)[x]$. $\psi_{p^k}$ is a ring homomorphism from $\Z(\zeta)[z]/\langle h(z)\rangle$ to $\Z(\zeta)[z]/\langle h(z), p,\zeta-1\rangle=\F_p[z]/\langle \psi_{p^k}(h(z))\rangle$.
\end{cor}
$T_{\ell}^k/\langle p,\zeta-1\rangle$ being isomorphic to $\overline{T}_{\ell}$ is a special case of the corollary.

We also need the following lemma proving rank relations under the quotient map $\psi_{p^k}$.

\begin{lem}\label{lem:quoRank}
Let $A$ be a matrix with entries in $T_{\ell}^k$ then we have the following bound,
{\em$$\rank_{\Q(\zeta)} A  \ge \rank_{\F_p} \psi_{p^k}(A),$$}
where $\psi_{p^k}(A)$ is the matrix with entries in $\overline{T}_{\ell}$ obtained by applying $\psi_{p^k}$ to each entry of $A$.  
\end{lem}

We now define the Vandermonde matrices with entries in $\F_p[z]$ which was defined by Arsovski in \cite{arsovski2021padic} and whose $\F_p$-rank will help us lower bound the size of Kakeya Sets.
\begin{define}[The matrix $M_{m,n}$]
The matrix $M_{m,n}$ is a matrix over $\F_p[z]$ with its rows and columns indexed by points in $\{0,\hdots,m-1\}^n$. The $(u,v)\in \{0,\hdots,m-1\}^n\times \{0,\hdots,m-1\}^n$ entry is
$$M_{m,n}(u,v)=z^{\langle u, v\rangle}.$$
\end{define}

For a matrix $M$ with entries in $\F[z]$ we let $M (\text{mod }G(z))$ is the matrix with same entries but over the ring $\F[z]/\langle G(z)\rangle$. We let $M^\ell_{m,n}$ be the matrix $M_{m,n} (\text{mod }(z-1)^\ell)$. We will need the following bound on the $\F_p$-rank of these matrices.

\begin{lem}[Rank of $M^\ell_{m,n}$]\label{lem:1rank}
Let $m\ge \ell$ then the $\F_p$-rank of $M^\ell_{m,n}$ is at least 
{\em $$\text{rank}_{\F_p} M^\ell_{m,n} \ge \left\lceil\binom{\ell \lceil\log_p(\ell)\rceil^{-1} +n}{n}\right\rceil.$$}
\end{lem}

For a given set of elements $V\subseteq \{0,\hdots,m-1\}^{n}$ let $M_{m,n}(V)$ refer to the sub-matrix obtained by restricting to rows of $M_{m,n}$ corresponding to elements in $V$ (similarly for $M^\ell_{m,n}$). In particular, for any given $u\in \{0,\hdots,m-1\}^n$ we let $M_{\ell,n}(u)$ refer to the $u$'th row of the matrix.

We first want to show the $\F_p$-rank of $M^\ell_{m,n}$ can be explained by the rows indexed by elements $u\in \{0,\hdots,m-1\}^n$ such that at least one coordinate is non-zero  modulo $p$. To that end we let $V_{m,p}$ represent the set of elements in $\{0,\hdots,m-1\}^n$ such that at least one co-ordinate is non-zero modulo $p$.

\begin{lem}\label{lem:rankp}
The row-space of $\Coeff(M^\ell_{m,n})$ equals the row-space of the sub-matrix $\Coeff(M^\ell_{m,n}(V_{m,p}))$.
\end{lem}
\begin{proof}
Consider the matrix $\Coeff(M^\ell_{m,n})$. For any $u\in \{0,\hdots,m-1\}^n$, the $u$'th row in $M^\ell_{m,n}$ will correspond to a $p^\ell$ block of rows $B_u$ in $\Coeff(M^\ell_{m,n})$. 
Say $u$ doesn't have a non-zero coordinate modulo $p$. We can find an element $u' \in V_{m,p}$ such that for some $i$, $p^iu'=u$. We claim that the block $B_u$ can be generated by the block $B_{u'}$ via a linear map. This follows because given the coefficient vector of a polynomial $Q(z)$ in $\overline{T}_{\ell}$ the coefficient vector of the polynomial $Q(z^p)\in \overline{T}_\ell$ can be obtained via a $\F_p$-linear map.
\end{proof}

We need to consider row vectors which encode the evaluation of monomials and their derivatives over points on a line in direction $u\in (\Z/p^k\Z)^{n}$. 

\begin{define}[The evaluation vector $U_{d}^{(\balpha)}(y)$]
Let $\F$ be a field, $n,d\in \Z$ and $\balpha\in \Z_{\ge 0}^n$. For any given point $y\in \F^n$ we define $U_d^{(\balpha)}(y)$ to be a row vector of size $d^n$ whose columns are indexed by monomials $Q=x_1^{j_1}x_2^{j_2}\hdots x_n^{j_n}\in \F[x_1,\hdots,x_n]$ for $j_k\in \{0,\hdots,d-1\},k\in [n]$ such that its $Q$'th column is $Q^{(\balpha)}(y)$. 
\end{define}

Let $L\subseteq (\Z/p^k\Z)^n$ be a line in direction $u\in \P (\Z/p^k\Z)^{n-1}$. A special case of the next lemma proves that for any polynomial $f\in \Z[x_1,\hdots,x_n]$ we can evaluate $f(z^{u})\in \overline{T}_{p^k}=\F_p[z]/\langle (z-1)^{p^k}\rangle,z^{u}=(z^{u_1},\hdots,z^{u_n})$ from the evaluation of $f$ on the points $f(\zeta^x),x\in L$. Let $\pi: L \rightarrow \Z_{\ge 0}$ be a function on the line such that $\sum_{x\in L}\pi(x)\ge \ell$. The next lemma states that we can decode $f(z^{u'})\in \overline{T}_{\ell}=\F_p[z]/\langle (z-1)^{\ell}\rangle$ from the evaluations of weight at most $\pi(\zeta^x)$ Hasse derivatives of $f$ at $\zeta^x,x\in L$ for any $u'$ in $\Z^n$ such that $u'\text{ (mod }p^k\text{)}=u$.

The lemma below can be thought of as analogous to how over finite fields the evaluation of a polynomial and its Hasse derivatives with high enough weight along a line in direction $u$ can be used to decode the evaluation of that polynomial at the point at infinity along $u$~\cite{DKSS13}.

\begin{lem}[Decoding from evaluations on rich lines]\label{lem:deRich}
Let $L=\{a+\lambda u| \lambda \in \Z/p^k\Z\}\subset (\Z/p^k\Z)^n$ with $a\in (\Z/p^k\Z)^n,u\in \P (\Z/p^k\Z)^{n-1}$, $u'\in \Z^n$ be such that {\em $u'\text{ (mod }p^k\text{)}=u$} and $\pi:L\rightarrow \Z_{\ge 0}$ be a function which satisfies $\sum_{x\in L} \pi(x) \ge \ell$.

Then, there exists elements $c_{\lambda,\balpha} \in \Q(\zeta)[z]$ (depending on $\pi,L$ and $u'$) for $\lambda \in \Z/p^k\Z$ and $\balpha\in  \Z_{\ge 0}^n$ with $\text{wt}(\balpha)< \pi(a+\lambda u)$ such that the following holds for all polynomials $f\in \Z[x_1,\hdots,x_n]$,

{\em$$ \psi_{p^k}\left(\sum\limits_{\lambda =0}^{p^k-1}\sum\limits_{\text{wt}(\balpha) < \pi(a+\lambda u) }c_{\lambda,\balpha} f^{(\balpha)}(\zeta^{a+\lambda u})\right)= f(z^{u'})\in \overline{T}_\ell .$$}
\end{lem}
For the application of $\psi_{p^k}$ in the statement of the lemma to make sense we must have its input be an element in $T_\ell^k$. In other words, we need the input of $\psi_{p^k}$ to be a polynomial in $z$ with coefficients in $\Z(\zeta)$. This is indeed the case. This lemma is basically Lemma 3.11 from \cite{DharGeneral} and the proof is near identical. For completeness we give the proof in the appendix.

We will need this Lemma in the form of the following Corollary.

\begin{cor}[Decoding $M^\ell_{m,n}$ from $U^{(\balpha)}_{m}$]\label{cor:decoding}
Let $L=\{a+\lambda u| \lambda \in \Z/p^k\Z\}\subset (\Z/p^k\Z)^n$ with $a\in (\Z/p^k\Z)^n,u\in \P (\Z/p^k\Z)^{n-1}$, $u'\in \Z^n$ be such that {\em $u'\text{ (mod }p^k\text{)}=u$} and $\pi:L\rightarrow \Z_{\ge 0}$ be a function which satisfies $\sum_{x\in L} \pi(x) \ge \ell$ then there exists a $\Q(\zeta)[z]$-linear combination (with coefficients depending on $\pi,L$ and $u'$) of the vectors $U_{m}^{(\balpha)}(\zeta^x)$ for $x\in L$ and $\balpha$ of weight strictly less than $\pi(x)$ which under the map $\psi_{p^k}$ gives us the vector $M^\ell_{m,n}(u')$.

This also implies that there exists for each $i=0,\hdots,\ell-1$ many $\Q(\zeta)$-linear combinations of $U_{m}^{(\balpha)}(\zeta^x)$ for $x\in L$ and $\balpha<\pi(x)$ such that under the map $\psi_{p^k}$ we get the $i$th row of {\em $\Coeff(M^\ell_{m,n}(u'))$}.
\end{cor}

To prove maximal bounds we will need to treat $M^\ell_{m,n}$ over $\overline{T}_\ell$ as a sub-matrix of $M^{\ell+1}_{m,n}$ over $\overline{T}_{\ell+1}$ in some sense. We make this precise now.

\begin{lem}
The row-space over $\F_p$ of {\em$\Coeff(M_{\ell,n})$} is a sub-space of the row-space over $\F_p$ of {\em $\Coeff(M_{\ell+1,n})$}. 
\end{lem}
\begin{proof}
This follows from the fact the coefficients of $Q$ modulo $(z-1)^\ell$ can be computed as a linear combination of the coefficients of $Q$ modulo $(z-1)^{\ell+1}$. This is because $(z-1)^\ell$ is a factor of $(z-1)^{\ell+1}$. 
\end{proof}

The previous lemma easily implies the next corollary using simple linear algebra and Lemmas~\ref{lem:1rank} and \ref{lem:rankp}.
\begin{cor}\label{cor:splitM}
We have a set $A=A_1\cup\hdots\cup A_{\ell}$ of linearly independent row vectors in the row space of {\em $\Coeff(M^\ell_{m,n})$} such that
\begin{enumerate}
  \item $$|A|=\left\lceil\binom{\ell \lceil\log_p(\ell)\rceil^{-1} +n}{n}\right\rceil,$$
    \item elements in $A_i$ are a subset of the rows of  {\em$\Coeff(M^i_{m,n}(V_{m,p}))$} and do note belong in the row space of  {\em$\Coeff(M^{i-1}_{m,n})$},
    \item and 
    $$|A_i|=\left\lceil\binom{i \lceil\log_p(i)\rceil^{-1} +n}{n}\right\rceil-\left\lceil\binom{(i-1) \lceil\log_p(i)\rceil^{-1} +n}{n}\right\rceil$$
    for $i\in [\ell]$.
\end{enumerate}
\end{cor}

In the above corollary, for the purposes of our arguments, $A_i$ in some sense encapsulates the linearly independent evaluations of degree $i$ monomials when we embed $\Z/p^k\Z$ over the complex torus and use the $\psi_{p^k}$ operation. This kind of decomposition is important for us to get Maximal Kakeya bounds.
\section{Warm Up: Maximal Kakeya bounds over finite fields}\label{sec-field}
In this section we prove Theorem \ref{thm-field}. We restate it here.
\begin{repthm}{thm-field}[Maximal Kakeya bounds over finite fields]
Let $n>0$ be an integer, $q$ a prime power. For any function $f:\F_q^n\rightarrow \C$ we have the following bound,
$$\sum\limits_{x\in \F_q^n} |f(x)|^n \ge \frac{1}{(2-1/q)^n} |\P \F_q^{n-1}|^{-1} \left(\sum\limits_{u\in \P \F_q^{n-1}} |f^*(u)|^n\right).$$
\end{repthm}
By a limiting and scaling argument we note that it suffices to prove the statement above only for natural number valued functions.

We first give a brief sketch for a weaker bound for the setting where $f$ is the indicator function of a set to give intuition. The proof will be a direction generalization of the proof of the Kakeya set problem via the polynomial method. First let us describe the proof in~\cite{dvir2009size} for the Kakeya set problem. That paper corresponds to the setting where the function $f:\F_q^n\rightarrow \N$ is the indicator function of a set $S\subseteq \F_q^n$ such that $f^*$ is always $q$. That is $S$ contains a line in every direction. 

Dvir's proof works by assuming the set has size smaller than $\binom{q-1+n}{n}\ge q^n/n!$ which is the number of degree at most $q-1$ monomials in $n$ variables. This means we can find a non-zero linear combination of these monomials to construct a polynomial $Q$ of degree at most $q-1$ which vanishes on the entirety of $S$. As $S$ contains a line in each direction $Q$ can be restricted on these lines and will vanish on $q$ points while having degree at most $q-1$. By elementary algebra we can observe that the highest degree part of $Q$ vanishes on every direction. This leads to a contradiction as a homogenous polynomial of degree at most $q-1$ can not vanish over every direction by the Schwartz-Zipple lemma~\cite{schwartz1979probabilistic,ZippelPaper}.

We modify this proof using Lemma~\ref{lem-goodmono} to find a more suitable polynomial. We now start out with an arbitrary set $S$ and let $f$ be its indicator function. Clearly, $f^*\le q$. Let 
$$B_{\ge k}=\{u\in \P \F_q^{n-1}| f^*(u)\ge k\},$$ 
in other words it is the set of directions for which there exists a line in that direction with intersection at least $k$ with $S$. We also let $\epsilon_{\ge k}=|B_{\ge k}|/|\P \F_q^{n-1}|$. 

We will prove that,
\begin{equation}\label{eq-setfield}
|S| \ge \sum\limits_{k=0}^{q-1} \epsilon_{\ge k+1}\binom{k+n-1}{n-1}.
\end{equation}
Rearranging the RHS will show that the above inequality gives us the desired bounds but with worse constants. For each $0\le k\le q-1$ we can find a set of monomials $P(k)$ of degree $k$ such that any of their linear combinations do not vanish over $B_{\ge k+1}$.

If we say \eqref{eq-setfield} does not hold then we can find a non-zero polynomial $Q$ formed by taking a linear combination of monomials from $\bigcup_{k=0}^{q-1} P(k)$ which vanishes on $S$. Say $Q$ is of degree $d$. Then consider the set of directions $B_{\ge d+1}$. For each direction in $B_{\ge d+1}$ there exists a line in that direction such that $Q$ vanishes on $d+1$ of its points. As $Q$ is of degree $d$ again by elementary algebra we can say that the degree $d$ part of $Q$ then must vanish on the direction vector. But the degree $d$ part of $Q$ is a linear combination of monomials from $P(d)$ and hence can't vanish on the entirety of $B_{\ge d+1}$ leading to a contradiction.

The general statement is proven by letting derivatives also vanish and a more careful analysis to get better constants.

\begin{proof}
 Let $f: \F_q^n\rightarrow \N$ and we define $B_{\ge k}=\{u||f^*(u)|\ge k,u\in \P \F_q^{n-1}\}$ and $B_{k}=\{u||f^*(u)|=k,u\in \P \F_q^{n-1}\}$. In other words, $B_{\ge k}$ is the set of directions where the maximal function is at least $k$ (similarly for $B_k$). We also let $w=\max_u f^*(u)$. For $m\in \N$ and $0\le k\le w$, let
$$\epsilon_{\ge k}=|B_{\ge k}|/|\P \F_q^{n-1}|,$$
$$\epsilon_{k}=|B_k|/|\P \F_q^{n-1}|$$
$$r_k = \left\lceil \frac{mk+1}{2q-1} -1\right\rceil,$$
and
$$d_k=r_kq-1.$$
By definition, $\epsilon_{\ge k}=\sum\limits_{i=k}^w \epsilon_k$. We also note that if $m\ge 2q$ then $0=r_0<r_1<r_2<\hdots<r_w$. The same also holds for $-1=d_0<d_1<\hdots<d_w$.

We will prove the following inequality for all $m\ge 2q$,
\begin{equation}\label{eq-field}
    \sum\limits_{i=1}^{w} \epsilon_{\ge i} \sum\limits_{j=d_{i-1}+1}^{d_i}\binom{j+n-1}{n-1} \le \sum\limits_{x\in \F_q^n} \binom{m f(x)+n-1}{n}.
\end{equation}
We claim \eqref{eq-field} proves the Theorem. Let us prove that first. The LHS of \eqref{eq-field} can be re-arranged to give,
\begin{align*}\label{eq-field2}
\sum\limits_{i=1}^{w} \epsilon_{\ge i} \sum\limits_{j=d_{i-1}+1}^{d_i}\binom{j+n-1}{n-1}&=\sum\limits_{i=1}^w \sum_{k=i}^w \sum\limits_{j=d_{i-1}+1}^{d_i} \epsilon_k\binom{j+n-1}{n-1}\\
&=\sum_{k=1}^w \sum\limits_{i=0}^{d_k}  \epsilon_k \binom{i+n-1}{n-1}=\sum_{k=1}^w \epsilon_k \binom{d_k+n}{n}.\numberthis
\end{align*}

We note as we let $m$ grow to $\infty$ we have,
$$ \lim\limits_{m\rightarrow \infty} \frac{1}{m^n}\binom{d_k+n}{n}= \frac{q^nk^n}{(2q-1)^n} \frac{1}{n!},$$
and
$$ \lim\limits_{m\rightarrow \infty} \frac{1}{m^n}\binom{mf(x)-1+n}{n}= f(x)^n \frac{1}{n!}.$$
These two limits with \eqref{eq-field} and \eqref{eq-field2} imply,
$$\sum\limits_{k=0}^w \epsilon_k q^n k^n \frac{1}{(2q-1)^n} \le \sum\limits_{x\in\F_q^n} f(x)^n.$$
As $q^n/(2q-1)^n=1/(2-1/q)^n$ and $\epsilon_k=|B_k|/|\P \F_q^{n-1}|$ where $B_k$ is precisely the set of directions for which $f^*$ is $k$ we are done.

We now just have to prove \eqref{eq-field} holds. For $j\in [d_{i-1}+1,d_i]$, using Lemma~\ref{lem-goodmono} we pick a set of monomials $P(j)$ of degree $j$ of size $\epsilon_{\ge i}\binom{j+n-1}{n-1}$ such that any non-zero linear combination of those don't vanish on the entirety of $B_{\ge i}$ with multiplicity at least $r_i$.

Say \eqref{eq-field} doesn't hold, then we can find a non-zero polynomial $Q$ formed by taking a non-zero linear combination of monomials in $\bigcup\limits_{j=0}^{d_w} P(j)$ such that $Q$ vanishes on $x\in \F_q^n$ with multiplicity at least $mf(x)$. This is because vanishing with multiplicity at least $mf(x)$ gives us $\binom{mf(x)+n-1}{n}$ many linear constraints and we have $\sum\limits_{j=0}^w |P(j)|$ many coefficients. $\sum\limits_{j=0}^w |P(j)|$ equals the LHS of \eqref{eq-field}.

Say $Q$ is of degree $d$ with $d_{i-1}<d\le d_{i}$ for some $i\le w$. Let $Q^H$ be the highest degree part of $Q$. In other words, $Q^H$ is the degree $d$ part of $Q$. By construction $Q^H$ is a non-zero linear combination of monomials in $P(d)$.

\begin{claim}
For all $u\in B_{\ge i}$,
$$\textsf{mult}(Q^H,u) \ge r_i.$$
\end{claim}
\begin{proof}
Let $\bj \in \Z_{\ge 0}^n$ with $\text{wt}(\bj)<r_i$. $Q^{(\bj)}$ is a degree $d-\text{wt}(\bj)$ polynomial. If we consider a $u\in B_{\ge i}$ then there exists a line $L_u=\{a+tu|t\in\F_q\}$ such that $\sum\limits_{x\in L_u} f(x) \ge i$. $Q^{(\bj)}$ will vanish on each point $x$ with multiplicity at least $mf(x)-\text{wt}(\bj)$. By Lemma~\ref{lem:multComp} we see that $Q^{(\bj)}(a+tu)$ vanishes with multiplicity at least $mf(a+tu)-\text{wt}(\bj)$ for $t\in \F_q$. As $Q^{(\bj)}(a+tu)$ is a degree at most $d-\text{wt}(\bj)$ polynomial we see that if $Q^{(\bj)}(a+tu)$ is non-zero then,
$$d-\text{wt}(\bj)\ge \sum\limits_{t\in \F_q} (f(a+tu)-\text{wt}(\bj))\ge mi-\text{wt}(\bj)q.$$
As $d\le d_i=qr_i-1$ and $\text{wt}(\bj)<r_i$ we have,
$$qr_i-1+(q-1)r_i>mi.$$
This means 
$$\left\lceil \frac{mk+1}{2q-1} -1\right\rceil=r_i>(mi+1)/(2q-1),$$
leading to a contradiction. This implies that $Q^{(\bj)}(a+tu)$ is identically $0$. We note the coefficient of $t^{d-\text{wt}(\bj)}$ in $Q^{(\bj)}(a+tu)$ is exactly $(Q^H)^{(\bj)}(u)$. As this argument holds for all $\bj \in \Z_{\ge 0}^n$ with $\text{wt}(\bj)<r_i$ we have that $Q^H$ vanishes on $u$ with multiplicity $r_i$.
\end{proof}

The previous claim leads to a contradiction as by construction $Q^H$ is a non-zero linear combination of monomials in $P(d)$ which by Lemma~\ref{lem-goodmono} cannot vanish on the entirety of $B_{\ge i}$ with multiplicity at least $r_i$.
\end{proof}

\section{Base case: Maximal Kakeya bounds over $\Z/p^k\Z$}\label{sec-base}
The proof here involves replicating the arguments in the proof of Theorem~\ref{thm-field} while applying polynomial method arguments over the complex torus.

\begin{repthm}{thm-maxpk}[Maximal Kakeya bounds over $\Z/p^k\Z$]
Let $n>0$ be an integer, $p$ prime and $k\in \N$. For any function $f:(\Z/p^k\Z)^n\rightarrow \N$ we have the following bound,
$$\sum\limits_{x\in (\Z/p^k\Z)^n} |f(x)|^n \ge \frac{C_{p^k,n}}{|\P(\Z/p^k\Z)^{n-1}|} \left(\sum\limits_{u\in \P (\Z/p^k\Z)^{n-1}} |f^*(u)|^n\right),$$
where
$$C_{p^k,n}= \frac{1}{(2(\lceil\log_p(\max_u f^*(u))+ \log_{p}(n)\rceil ))^n}.$$
When $p>n$ we can improve $C_{p^k,n}$ to,
$$C_{p^k,n}=(\lceil\log_p(\max_u f^*(u))\rceil+1)^{-n} (1+n/p)^{-n}.$$
\end{repthm}
\begin{proof}
 Let $f: (\Z/p^k\Z)^n\rightarrow \N$ and we define $B_{\ge v}=\{u||f^*(u)|\ge v,u\in \P (\Z/p^k\Z)^{n-1}\}$ and $B_{v}=\{u||f^*(u)|=v,u\in \P (\Z/p^k\Z)^{n-1}\}$. In other words, $B_{\ge v}$ is the set of directions where the maximal function is at least $v$ (similarly for $B_v$). We also let $w=\max_u f^*(u)$. For $m\in \N$ and $0\le v\le w$, let
$$\epsilon_{\ge v}=|B_{\ge v}|/|\P (\Z/p^k\Z)^{n-1}|,$$
$$\epsilon_{v}=|B_v|/|\P (\Z/p^k\Z)^{n-1}|$$
$$r_v = mv.$$
By definition, $\epsilon_{\ge v}=\sum\limits_{i=v}^w \epsilon_v$. 

We prove the following for all $m\in \N$.
\begin{align*}\label{eq-pk}
    & \sum\limits_{i=1}^{w} \epsilon_{\ge i} \sum\limits_{j=r_{i-1}+1}^{r_i}\left(\left\lceil\binom{j\lceil\log_p(j)\rceil^{-1} +n}{n}\right\rceil-\left\lceil\binom{(j-1)\lceil\log_p(j-1)\rceil^{-1} +n}{n}\right\rceil\right)\\
    \le & \sum\limits_{x\in (\Z/p^k\Z)^n} \binom{m f(x)+n-1}{n}.\numberthis
\end{align*}
We claim \eqref{eq-pk} proves the Theorem. Let us prove that first. The LHS of \eqref{eq-pk} can be re-arranged to give,
\begin{align*}\label{eq-pk2}
&\sum\limits_{i=1}^{w} \epsilon_{\ge i} \sum\limits_{j=r_{i-1}+1}^{r_i}\left(\left\lceil\binom{j\lceil\log_p(j)\rceil^{-1} +n}{n}\right\rceil-\left\lceil\binom{(j-1)\lceil\log_p(j-1)\rceil^{-1} +n}{n}\right\rceil\right)\\
=&\sum\limits_{i=1}^w \sum_{v=i}^w \sum\limits_{j=r_{i-1}+1}^{r_i} \epsilon_v\left(\left\lceil\binom{j\lceil\log_p(j)\rceil^{-1} +n}{n}\right\rceil-\left\lceil\binom{(j-1)\lceil\log_p(j-1)\rceil^{-1} +n}{n}\right\rceil\right)\\
=&\sum_{v=1}^w \sum\limits_{i=0}^{r_v}  \epsilon_v\left(\left\lceil\binom{j\lceil\log_p(j)\rceil^{-1} +n}{n}\right\rceil-\left\lceil\binom{(j-1)\lceil\log_p(j-1)\rceil^{-1} +n}{n}\right\rceil\right)\\
=&\sum_{v=1}^w \epsilon_v \left\lceil\binom{r_v\lceil\log_p(r_v)\rceil^{-1} +n}{n}\right\rceil. \numberthis
\end{align*}

If we set $m=n$ then using \eqref{eq-pk} and \eqref{eq-pk2} we have,
\begin{align*}
\frac{1}{n!}\sum_{v=1}^w \epsilon_v \frac{v^n}{\lceil\log_p(v)+\log_p(n)\rceil^n}\le n^{-n}\sum_{v=1}^w \epsilon_v \left\lceil\binom{r_v\lceil\log_p(r_v)\rceil^{-1} +n}{n}\right\rceil\le \sum\limits_{x\in (\Z/p^k\Z)^n} \frac{1}{n!}\prod\limits_{i=0}^{n-1} (f(x)+i/n)^n.
\end{align*}
As $(f(x)+i/n)\le 2f(x)$ for integer $f(x)>0$ and $i<n$ we are done.
If we set $m=p$ a similar calculation and using the fact that $(f(x)+i/n)\le f(x)(1+(n-1)/p)$ for integer $f(x)>0$ and $i<n$ we will be done.

Now it remains to prove \eqref{eq-pk} for all $m\in \N$.

Fix some $m\in \N$. Consider a random matrix $G$ uniformly chosen from $\GL_n(\Z/p^k\Z)$ and let $G\cdot f$ be the function which maps $x$ to $f(G^{-1}x)$. For a general set $S$ if we let $G\cdot S$ be the set $\{G\cdot x|x\in S\}$ then the following identities are immediate,
$$G\cdot B_{\ge v}= \{u|(G\cdot f)^*(u)|=|f^*(G^{-1}u)|\ge v,u\in \P (\Z/p^k\Z)^{n-1}\}$$
and 
$$G\cdot B_{v}=\{u||(G\cdot f)^*(u)|=v,u\in \P (\Z/p^k\Z)^{n-1}\}.$$

Consider the matrix $U$ composed of the row vectors $U_{mw}^{\balpha}(\zeta^x)$ for $x\in (\Z/p^k\Z)^n$ and $\text{wt}(\balpha) < m(G\cdot f)(x)$. This has rank at most,
\begin{align}\label{eq:baseLHS}
    \sum\limits_{x\in (\Z/p^k\Z)^n} \binom{m f(x)+n-1}{n}.
\end{align}

Recall, $V_{mw,p}$ is the set of vectors in $\{0,\hdots,mw-1\}^n$ such that it at least has one non-zero coordinate modulo $p$.

Consider the function $m(G\cdot f)$. For any $v=1,\hdots,w$ by construction we have that for every $u \in G \cdot B_{\ge v}$ there exists a line $L_u$ such that $(G\cdot f)(x)\ge v$. Now, Corollary~\ref{cor:decoding} implies that for any $u \in G \cdot B_{\ge v}$, $u'\in V_{mw,p},u' \text{ (mod }p^k\text{)}=u$, $j=r_{v-1}+1,\hdots,r_v$ and $i=0,\hdots,j-1$ there exists a $\Q(\zeta)$-linear combination of $U_{mw}^{\balpha}(\zeta^x)$ where $x\in L_u$ and $\text{wt}(\balpha)\le m(G\cdot f)(x)-1$ such that under $\psi_{p^k}$ we get the $i$th row of $\Coeff(M^{j}_{mw,n}(u'))$.

This means that there exists a matrix with $\Q(\zeta)$-entries such that multiplying this matrix to $U$ from the left and applying $\psi_{p^k}$ gives us a matrix $M$ whose rows are composed of the rows in $\Coeff(M^{j}_{mw,n}(u'))$ for $v=1,\hdots,w$, $u \in G \cdot B_{\ge v}$, $u'\in V_{mw,p},u' \text{ (mod }p^k\text{)}=u$, $j=r_{v-1}+1,\hdots,r_v$ and $i=0,\hdots,j-1$. By Lemma~\ref{lem:quoRank} we have that \eqref{eq:baseLHS} is lower bounded by the rank of $M$. As $G$ is chosen randomly we will show the expected value of the rank of $M$ is at least \begin{equation}\label{eq:baseRHS}
    \sum\limits_{i=1}^{w} \epsilon_{\ge i} \sum\limits_{j=r_{i-1}+1}^{r_i}\left(\left\lceil\binom{j\lceil\log_p(j)\rceil^{-1} +n}{n}\right\rceil-\left\lceil\binom{(j-1)\lceil\log_p(j-1)\rceil^{-1} +n}{n}\right\rceil\right).
\end{equation}
This will complete the proof.

By Corollary~\ref{cor:splitM} we have sets $A_1,\hdots,A_{r_w}$ of linearly independent row vectors in the row space of  $\Coeff(M^{mw}_{mw,n})$ such that elements in $A_i$ are a subset of the rows in $\Coeff(M^{i}_{mw,n}(V_{mw,p}))$ and do not belong in the row space of $\Coeff(M^{i-1}_{mw,n}(V_{mw,p}))$. We also have that,
$$|A_i|=\left\lceil\binom{i \lceil\log_p(i)\rceil^{-1} +n}{n}\right\rceil-\left\lceil\binom{(i-1) \lceil\log_p(i)\rceil^{-1} +n}{n}\right\rceil$$
for $i\in [r_w]$.
\begin{claim}
For $j=r_{i-1}+1,\hdots,r_i$ any given row in $A_j$ will appear in $M$ with probability at least $\epsilon_{\ge i}$
\end{claim}
\begin{proof}
Let $j$ be some number from $r_{i-1}+1,\hdots,r_i$. We note that any given row $y$ in $A_j$ will be from some $\Coeff(M^{j-1}_{mw,n}(u'))$ where $u'\in V_{mw,p}$. Now, if $G\cdot B_{\ge i}$ contains $u' \text{ (mod }p^k\text{)}$ then $y$ will appear in the matrix $M$. As $\GL_n(\Z/p^k\Z)$ acts transitively on $\P (\Z/p^k\Z)^{n-1}$ this will happen with probability at least $\epsilon_{\ge i}$.
\end{proof}
Using the linearity of expectation, the above claim immediately implies that the expected rank of $M$ is at least \eqref{eq:baseRHS}.
\end{proof}
\section{Induction Step: Maximal Kakeya bounds over $\Z/N\Z$}\label{sec-induct}

We first give an outline of the proof in a special case with worse constants to highlight the overall strategy. The goal is to extend the proof from \cite{dhar2021proof} using the probabilistic method and Corollary~\ref{cor:splitM} to prove Theorem~\ref{thm-maxN}.

Let $N=p^{k_0}q^{k_1},N_0=p^{k_0},N_1=q^{k_1}$ where $p$ and $q$ are distinct primes. Given a set $S\subseteq (\Z/N\Z)^n$ let $f:(\Z/N\Z)^n\rightarrow \Z_{\ge 0}$ be its indicator function. Let $L_u$ be the line in direction $u \in \P (\Z/N\Z)^{n-1}$ such that $f^*(u)=\sum_{x\in L(u)} f(u)=|S\cap L(u)|$. In other words, $L(u)$ is the line which has the maximal intersection with $S$ in direction $u$. $L_u$ can decomposed as a product of lines $L_0(u_0,u_1)\subseteq (\Z/N_0\Z)^n$ and $L_1(u_0,u_1)\subseteq (\Z/N_1\Z)^n$ in directions $u_0 \in \P (\Z/N_0\Z)^{n-1}$ and $u_1\in \P (\Z/N_1\Z)^{n-1}$ respectively. Note, that $L_0(u_0,u_1)$ and $L_0(u_0,u_1)$ can and do depend on $u=(u_0,u_1)$ and not just on $u_0$ or $u_1$ respectively.

Let $\zeta$ be a primitive $N_0$'th root of unity in $\C$ and $\Ind_y$ be the indicator vector of a point $y\in (\Z/N_1\Z)^n$. We then examine the span of vectors 
$$U_{N_0}^{\bzero}(\zeta^x)\otimes \Ind_y$$
for $x\in L_0(u),y\in L_1(u),u\in \P (\Z/N\Z)^{n-1}$ and $\bzero = (0,\hdots,0)\in \Z_{\ge 0}^n$. By counting the number of vectors we see the dimension of the space spanned by these vectors is at most $|S|$.

For $(u_0,u_1)\in \P (\Z/N_0\Z)^{n-1}\times \P (\Z/N_1\Z)^{n-1}$, we define the function $g_{L(u_0,u_1)}:L_1(u_0,u_1)\rightarrow \Z_{\ge 0}$ as 
$$g_{L(u_0,u_1)}(y)= \sum_{x \in L_0(u_0,u_1)} f(x,y) = |S\cap L_0(u_0,u_1)\times \{y\}|.$$
We note that $g_{L(u_0,u_1)}\le N_0$. In other words, $g_{L(u_0,u_1)}$ is slicing the line $L(u_0,u_1)$ along the $\Z/N_0\Z$ part (that is along $u_0$) and measuring its intersection with $S$. For each $L(u)$ we let $B_{\ge i}(L(u))=\{y|g_{L(u)}(y)\ge i,y\in L_1(u)\}$ and $b_{\ge i}(L(u))=|B_{\ge i}(L(u))|$.

Given $L(u_0,u_1)$ and $y\in L_1(u_0,u_1)$, we can use Corollary~\ref{cor:decoding} to linearly generate the vector
$$M^{g_{L(u_0,u_1)}(y)}_{N_0,n}(u_0)\otimes \Ind_y $$
from the vectors $U_{m}^{\bzero}(\zeta^x)\otimes \Ind_y$ for $x\in L_0(u_0,u_1)$ after applying $\psi_{p^{k_0}}$. Now by Lemma~\ref{lem:quoRank} and Fact~\ref{lem:coeffR} the dimension of the space spanned by the rows in $\Coeff(M^{g_{L(u_0,u_1)}(y)}_{N_0,n}(u_0))\otimes \Ind_y$ for $(u_0,u_1)\in \P (\Z/N_0\Z)^{n-1}\times \P (\Z/N_1\Z)^{n-1}$, $y\in L_1(u_0,u_1)$ is at most $|S|$. Let us call this set of rows $M$.

Corollary~\ref{cor:splitM} will give us a number of linearly independent vectors $A_1,\hdots, A_{N_0}$ such that $A_i\subseteq \Coeff(M^{i}_{N_0,n}),i \in [N_0]$ and the row-spaces of $A_i$ form disjoint subspaces. Now, the goal is to use Fact~\ref{fact-tenLin} by taking each element $z$ in the set $A_i$ and collect $y\in (\Z/N_1\Z)^n$ such that $\Coeff(M^{g_{L(u_0,u_1)}(y)}_{N_0,n}(u_0))$ contains $z$. This will give us a set of linearly independent elements from $M$ and their number will lower bound $|S|$. Note if we know for which $i\in [N_0]$ and $u_0 \in \P (\Z/N_0\Z)^{n-1}$ a given $z$ is in $\Coeff^{~i}_{N_0,n}(u_0))$ then we just need to figure out for what $y$ there exists a $u_1 \in \P (\Z/N_1\Z)^{n-1}$ such that $g_{L(u_0,u_1)}(y)=i$. 

With the above in mind, let $S(u_0,i)$ for $u_0\in \P N_0^{n-1}$ and $i\in [N_0]$ be the set $$S(u_0,i)=\bigcup\limits_{u_1\in \P (\Z/N_1\Z)^{n-1}}B_{\ge i}(L(u_0,u_1)).$$
Using Theorem~\ref{thm-maxpk} (in general here we use the induction hypothesis) we have,
\begin{equation*}
    \rank_{\F_{p_0}}\{\Ind_y|y\in S(u_0,i)\} =|S(u_0,i)|\ge \sum\limits_{u_1 \in \P (\Z/N_1\Z)^{n-1}} \frac{C_{N_1,n}b_{\ge i}(L(u_0,u_1))^n}{|\P (\Z/N_1\Z)^{n-1}|}, 
\end{equation*}
where $C_{N_1,n}= 2^{-1}(\lceil k_1 + \log_q(n)\rceil)^{-n}$. Now we see Fact~\ref{fact-tenLin} and Corollary~\ref{cor:splitM} implies that,
$$ \sum\limits_{u_0 \in \P (\Z/N_0\Z)^{n-1}} \sum_{j\in [N_0]} |S(u_0, j)| |\Coeff(M^j_{N_0,n}(u_0))\cap A_j| \le |S|.$$

Ideally, the above would be enough to complete the proof but it is not easy to directly argue the above is large. What we do is apply random rotations on the $\Z/N_0\Z$ component of our function which will lead to a random permutation of the rows $M^{g_{L(u_0,u_1)}(y)}_{N_0,n}(u_0)$ along the $u_0$ co-ordinate. We then take expectation of the above along with a very simple partition argument to get the desired bounds. For the general case and better constants we use multiplicities as in the proof of Theorem~\ref{thm-maxpk}.

\begin{repthm}{thm-maxN}[Maximal Kakeya bounds over $\Z/N\Z$]
Let $n>0$ be an integer and $N=p_1^{k_1}\hdots p_r^{k_r}$ with $p_i$ primes and $k_i\in \N$. For any function $f:(\Z/N\Z)^n\rightarrow \N$ we have the following bound,
$$\sum\limits_{x\in (\Z/N\Z)^n} |f(x)|^n \ge \frac{C_{N,n}}{|\P(\Z/N\Z)^{n-1}|} \left(\sum\limits_{u\in \P (\Z/N\Z)^{n-1}} |f^*(u)|^n\right),$$
where
{\em \begin{align*}
    C_{N,n}=&\left(\frac{1}{2(\log(\mweight(f,p_1))+1)\lceil \log_{p_1}(\mweight(f,p_1))+\log_{p_1}(n)\rceil}\right)^n\\
    &\cdot \left(\frac{1}{2(k_r+\lceil \log_{p_r}(n)\rceil)}  \prod\limits_{i=2}^{r-1}\frac{1}{2(k_i\log(p_i)+1)(k_i+\lceil \log_{p_i}(n)\rceil)}\right)^n
\end{align*}.}
\end{repthm}
\begin{proof}
The $r=1$ case is Theorem~\ref{thm-maxpk}. Let the theorem be true for some $r$. We will now prove it for $N=p_0^{k_0}p_1^{k_1}\hdots p_r^{k_r}=p_0^{k_0}N_1$. Let $f:(\Z/N\Z)^n\rightarrow \N$ and $w=\textsf{mweight}(f,p_0)$. We will use Fact~\ref{fact:geoChine} to decompose objects into a $\Z/p_0^{k_0}\Z$ component and $\Z/N_1\Z$ component.

Let $L(u)$ be the line in direction $u\in \P (\Z/N\Z)^{n-1}$ such that $\sum_{x\in L(u)} f(x)=f^*(u)$. $L(u)=L(u_0,u_1)$ can be decomposed into a tensor product of lines $L_0(u)$ in direction $u_0\in \P(\Z/p_0^{k_0})^{n-1}$ and $L_1(u)$ in direction $u_1\in \P (\Z/N_1\Z)^{n-1}$. For each $L(u)$ we define the function $g_{L(u)}: L_1(u)\rightarrow \Z_{\ge 0}$ as follows $$g_{L(u)}(y)=\sum_{x \in L_0(u)} f(x,y).$$
By definition of $\mweight$ we have $g_{L(u)}\le w$. For each $L(u)$ we let $B_{\ge i}(L(u))=\{y|g_{L(u)}(y)\ge i,y\in L_1(u)\}$ and $b_{\ge i}(L(u))=|B_{\ge i}(L(u))|$. $B_{\ge i}(L(u))$ is the set of points $y$ on $L_1(u)$ which have a weight greater than $i$ on the $\Z/N_0\Z$ line $L_0(u)\times \{y\}$. We similarly define $B_{= i}(L(u))=\{y|g_{L(u)}(y)=i,y\in L_1(u)\}$ and $b_{= i}(L(u))=|B_{=i}(L(u))|$.

Let $\zeta$ be a $p_0^{k_0}$'th root of unity. Consider a random matrix $G$ uniformly chosen from $\GL_n(\Z/p_0^{k_0}\Z)$. Consider the function $f_G(x,y)=f(G^{-1}\cdot x,y)$ where $x\in (\Z/p_0^{k_0}\Z)^n$ and $y\in (\Z/N_1\Z)^n$. For any set $S$ we also define $G\cdot S=\{G\cdot x| x\in S\}$ where $G$ acts on just the $\Z/p_0^{k_0}\Z$ part of $x$ if $x\in (\Z/N\Z)^n$ and is the usual action if $x\in (\Z/p_0^{k_0}\Z)^n$. We now consider the matrix $U_G$ composed of the row vectors 
$$U^{\balpha}_{mw}(\zeta^{G\cdot x})\otimes \Ind_y$$
for all $x\in (\Z/p_0^{k_0}\Z)^n$ and $y\in (\Z/N_1\Z)^n$ and $\balpha \in \Z_{\ge 0}^n$ with $\text{wt}(\balpha) < m f_G(x,y)$. The rank of $U_G$ is bounded above by
\begin{equation}
    \sum\limits_{x\in (\Z/N\Z)^n} \binom{mf(x)+n-1}{n}.\label{eq-maxNF}
\end{equation}

Recall, $V_{mw,p_0}$ is the set of vectors in $\{0,\hdots,mw-1\}^n$ such that it at least has one non-zero coordinate modulo $p_0$.

Consider the function $mf_G$. We see that by construction the line 
$$G\cdot L(u_0, u_1)= (G\cdot L_0(u_0,u_1), L_1(u_0,u_1))$$
will be a line in direction $(G \cdot u_0,u_1)$ such $\sum_{x \in G\cdot L(u_0, u_1)}f_G(x) = f^*(u_0,u_1) = f_G^*(G\cdot u_0,u_1)$.  Now, Corollary~\ref{cor:decoding} implies that for a $y\in L_1(u_0,u_1)$ and $u_0'\in V_{mw,p_0}, u_0'\text{(mod }p_0^{k_0}\text{)}=G\cdot u_0,j=1,\hdots,m g_{L(u_0,u_1)}(y),i=0,\hdots,j-1$ we have a $\Q(\zeta)$-linear combination of $U_{mw}^{\balpha}(\zeta^{x})$ where $x\in G\cdot L_0(u_0,u_1)$ and $\text{wt}(\balpha)\le m f_G(x,y)-1$ such that under $\psi_{p_0^{k_0}}$ we get the $i$th row of $\Coeff(M^j_{mw,n}(u_0'))$.

This means that there exists a matrix with $\Q(\zeta)$-entries such that multiplying this matrix to $U$ from the left and applying $\psi_{p_0^{k_0}}$ gives us a matrix $M_G$ whose rows are composed of the rows in $\Coeff(M^j_{mw,n}(u_0'))\otimes \Ind_y$ for $(u_0,u_1)\in \P (\Z/N\Z)^{n-1}$, $y\in L_1(u_0,u_1)$ and $u_0'\in V_{mw,p_0}, u_0'\text{(mod }p_0^{k_0}\text{)}=G\cdot u_0,$ and $j=1,\hdots,m g_{L(u_0,u_1)}(y)$.

By construction and Lemma~\ref{lem:quoRank}, the $\F_{p_0}$-rank of $M_G$ lower bounds the rank of $U_G$ and hence lower bounds \eqref{eq-maxNF}. In general, it suffices to lower bound the expected rank of $M_G$ over a random choice of $G\in \GL_n(\Z/p_0^{k_0}\Z)$. We do that in the next claim.

\begin{claim}\label{claim-maxNIne}
{\em $$\sum\limits_{u \in \P (\Z/N\Z)^{n-1},i\in [w]} \frac{C_{N_1,n}b_{\ge i}(L(u))^n}{|\P (\Z/N\Z)^{n-1}|} \cdot  \left(\left\lceil\binom{\frac{mi}{ \lceil\log_{p_0}(mi)\rceil} +n}{n}\right\rceil-\left\lceil\binom{\frac{m(i-1)}{ \lceil\log_{p_0}(m(i-1))\rceil} +n}{n}\right\rceil\right)$$}
is a lower bound for {\em $\E_{G\in \GL_n(\Z/p_0^{k_0}\Z)}[ \rank_{\F_{p_0}} M_G]$} where 
$$C_{N_1,n}=\left(\prod_{j=2}^{r-1}\frac{1}{k_j\log(p_j)+1}\prod\limits_{i=2}^r\frac{1}{(2(k_i+\lceil\log_{p_i}(n)\rceil))}\right)^n.$$
\end{claim}
\begin{proof}
Let $S(u_0,i)$ for $u_0\in \P (\Z/p_0^{k_0}\Z)^{n-1}$ and $i\in [w]$ be the set $$S(u_0,i)=\bigcup\limits_{u_1\in \P (\Z/N_1\Z)^{n-1}}B_{\ge i}(L(u_0,u_1)).$$
Using the induction hypothesis (in particular we use Theorem~\ref{thm-maxSet} for $r$ prime factors) we have,
\begin{equation}
    \rank_{\F_{p_0}}\{\Ind_y|y\in S(u_0,i)\} =|S(u_0,i)|\ge \sum\limits_{u_1 \in \P (\Z/N_1\Z)^{n-1}} \frac{C_{N_1,n}b_{\ge i}(L(u_0,u_1))^n}{|\P (\Z/N_1\Z)^{n-1}|}. \label{eq-Sibound}
\end{equation}

Recall for a fixed $G\in \GL_n(\Z/p_0^{k_0}\Z)$, $M_G$'s rows are the rows in $\Coeff(M^j_{mw,n}(u_0'))\otimes \Ind_y$ for $(u_0,u_1)\in \P (\Z/N\Z)^{n-1}$ ,$y\in L_1(u_0,u_1)$ and $u_0'\in V_{mw,p_0}, u_0'\text{(mod }p_0^{k_0}\text{)}=G\cdot u_0,$ and $j\in [m g_{L(u_0,u_1)}(y)]$.

We can easily re-write this condition to note that $M_G$'s rows are the rows in $\Coeff(M^j_{mw,n}(u_0'))\otimes \Ind_y$ for $u_0\in \P (\Z/p_0^{k_0}\Z)^{n-1}$, $u_0'\in V_{mw,p_0}, u_0'\text{(mod }p_0^{k_0}\text{)}=G\cdot u_0,$ and $y\in S(u_0,\lceil j/m \rceil)$ for $j\in [mw]$.

By Corollary~\ref{cor:splitM} we have a set $A=A_1\cup\hdots \cup A_{mw}$ of linearly independent row vectors in the row space of $\Coeff(M^{mw}_{mw,n})$ such that
$A_i$ are a subset of rows in $\Coeff(M^i_{mw,n}(V_{m,p_0}))$ and do not belong in the row space of $\Coeff(M^{i-1}_{mw,n})$ for $i\in [mw]$. We also have
$$|A_i|=\left\lceil\binom{i \lceil\log_{p_0}(i)\rceil^{-1} +n}{n}\right\rceil-\left\lceil\binom{(i-1) \lceil\log_{p_0}(i-1)\rceil^{-1} +n}{n}\right\rceil $$
for $i\in [mw]$.

For a $u_0 \in \P (\Z/p_0^{k_0}\Z)^{n-1}$, let $V'(u_0)$ be the set of $u_0'\in V_{mw,p_0}$ such that $u_0'\text{(mod }p_0^{k_0}\text{)}=u_0$.

Now Fact~\ref{fact-tenLin} implies that,
$$\sum\limits_{u_0 \in \P (\Z/p_0^{k_0}\Z)^{n-1}} \sum_{j\in [mw]} |S(u_0,\lceil j/m\rceil )| \left(\sum\limits_{u'_0\in V'(G\cdot u_0)}  |\Coeff(M^j_{mw,n}(u'_0)\cap A_j|\right) \le \rank_{\F_{p_0}} M_G.$$

By \eqref{eq-Sibound}, to prove the claim it now suffices to show that
$$\E_{G\in \GL_n(\Z/p_0^{k_0}\Z)}\left[\sum\limits_{u'_0\in V'(G\cdot u_0)}  |\Coeff(M^j_{mw,n}(u'_0)\cap A_j|\right]\ge \frac{|A_j|}{|\P (\Z/p_0^{k_0}\Z)^{n-1}|}.$$

A given row vector in $A_j$ corresponds to a row vector in $M^j_{mw,n}(V'(u'))$ for some $u'\in \P (\Z/p_0^{k_0}\Z)^{n-1}$. This means that $A_j$ appears in $\Coeff(M^j_{mw,n}(u'_0))$ for some $u'_0\in V'(G\cdot u_0)$ if and only if $G\cdot u_0 = u'$. As the action of $\GL_n(\Z/p_0^{k_0}\Z)$ is transitive on $\P (\Z/p_0^{k_0}\Z)^{n-1}$ we see that the probability of this happening is $|\P (\Z/p_0^{k_0}\Z)^{n-1}|^{-1}$. Linearity of expectation then completes the proof.
\end{proof}

We will also need another simple claim.
\begin{claim}\label{claim-srint}
For all $u$,
{\em $$\sum\limits_{i\in [w]} b_{\ge i}(L(u))^n (i^n-(i-1)^n)\ge \left(\frac{|f^*(u)|}{\log(w)+1}\right)^n.$$}
\end{claim}
\begin{proof}
By definition we have 
\begin{align}
    b_{\ge 1}(L(u))\ge b_{\ge 2}(L(u))\ge \hdots\ge b_{\ge w}(L(u)).\label{eq-clsr1}
\end{align}
We will show that there exists some $j\in [w]$ such that $b_{\ge j}(L(u))\ge f^*(u)/(j(\log(w)+1))$. Then using \eqref{eq-clsr1} we will be done.

Say for all $j\in [w]$ we have $b_{\ge j}(L(u))< f^*(u)/(j(\log(w)+1))$. By the choice of $L(u)$ we have 
$$\sum\limits_{i\in [w]} b_{=i}(L(u))i=\sum\limits_{i\in [w]} b_{\ge i}(L(u)) =  f^*(u).$$
But the above leads to a contradiction if we sum $b_{\ge i}(L(u))< f^*(u)/(i(\log(w)+1))$ for $i\in [w]$.
\end{proof}

We can re-arrange the summation in the lower bound of Claim~\ref{claim-maxNIne} to get,
\begin{align*}
    \sum\limits_{u \in \P (\Z/N\Z)^{n-1},i\in [w]} \frac{C_{N_1,n}b_{\ge i}(L(u))^n}{|\P (\Z/N\Z)^{n-1}|} \cdot  \left(\left\lceil\binom{\frac{mi}{ \lceil\log_{p_0}(mi)\rceil} +n}{n}\right\rceil-\left\lceil\binom{\frac{m(i-1)}{ \lceil\log_{p_0}(m(i-1))\rceil} +n}{n}\right\rceil\right) & \ge \\
    \sum\limits_{u \in \P (\Z/N\Z)^{n-1},i\in [w]} \frac{C_{N_1,n}}{|\P (\Z/N\Z)^{n-1}|} \cdot\left\lceil\binom{\frac{mi}{ \lceil\log_{p_0}(mi)\rceil} +n}{n}\right\rceil (b_{\ge i}(L(u))^n-b_{\ge i+1}(L(u))^n) & \ge \\
    \sum\limits_{u \in \P (\Z/N\Z)^{n-1},i\in [w]} \frac{C_{N_1,n}}{|\P (\Z/N\Z)^{n-1}|} \cdot \frac{(mi)^n}{\lceil\log_{p_0}(w)+\log_{p_0}(m)\rceil^n } (b_{\ge i}(L(u))^n-b_{\ge i+1}(L(u))^n)& \ge \\
    \sum\limits_{u \in \P (\Z/N\Z)^{n-1},i\in [w]} \frac{C_{N_1,n}}{|\P (\Z/N\Z)^{n-1}|} \cdot \frac{b_{\ge i}(L(u))^n}{\lceil\log_{p_0}(w)+\log_{p_0}(m)\rceil^n } ((mi)^n-(m(i-1))^n).
\end{align*}

Setting $m=n$ in the above inequality combined with Claim~\ref{claim-maxNIne} and the fact that $\E_{G\in \GL_n(\Z/p_0^{k_0}\Z)}[ \rank_{\F_{p_0}} M_G]$ lower bounds \eqref{eq-maxNF} gives us
$$\frac{1}{n!}\sum\limits_{u \in \P (\Z/N\Z)^{n-1},i\in [w]} \frac{C_{N_1,n}}{|\P (\Z/N\Z)^{n-1}|} \cdot  \frac{b_{\ge i}(L(u))(i^n-(i-1)^n)}{\lceil\log_{p_0}(w)+\log_{p_0}(n)\rceil^n} \le \sum\limits_{x\in (\Z/N\Z)^n} \frac{1}{n!} \prod\limits_{i=0}^{n-1} (f(x)+i/n).$$

Using the fact that $f(x)+i/n\le 2f(x)$ for $f(x)\ge 1$ and Claim~\ref{claim-srint} we are done.
\end{proof}

\section{Proof of Conjecture~\ref{conj}}\label{sec-conj}
\begin{repconjecture}{conj}[Kakeya Maximal conjecture over $\Z/N\Z$]
For all $\eps > 0$ and $n\in \N$ there exists a constant $C_{n,\eps}$ such that the following holds:
For a choice of a line $L(u)$ for each direction $u\in \P (\Z/N\Z)^{n-1}$ we have,
$$ \left\|\sum\limits_{u \in \P (\Z/N\Z)^{n-1}} \Ind_{L(u)}\right\|_{\ell^{n/(n-1)}} \leq C_{n,\epsilon} N^{\eps} \left(\sum\limits_{u \in \P (\Z/N\Z)^{n-1}} |L(u)| \right)^{(n-1)/n}.$$
\end{repconjecture}

We note $2N^{n-1}\ge \left(\sum_{u \in \P (\Z/N\Z)^{n-1}} |L(u)| \right)^{(n-1)/n}\ge N^{n-1}$ so it suffices to replace the right hand side by $N^{n-1}$. 

We will also need the following known number theoretic estimate which follows from Theorem 315 in \cite{Hardy2008AnIT}.

\begin{fact}\label{fact-Prest}
For a natural number $N$ with prime factorization $p_1^{k_1}\hdots p_r^{k_r}$ we have,
$$\prod\limits_{i=1}^r k_i\log(p_i) = N^{o(1)},$$
where $o(1)$ tends to $0$ as $N\rightarrow \infty$.
\end{fact}

\begin{proof}[Proof of Conjecture~\ref{conj}]
Throughout this proof we treat $n$ as fixed. In particular we hide any functions only depending on $n$ in the $O$ and $\Omega$ notation.

Let $N=p_1^{k_1}\hdots p_r^{k_r}$ for primes $k_1>\hdots>k_r$. 

Let $h(x)=\sum_{u \in \P (\Z/N\Z)^{n-1}} \Ind_{L(u)}(x)$ and $g(x)= h(x)^{1/(n-1)}$. We note $\|g\|_\infty\le 2N$.

We immediately have the following identity,
\begin{equation}
    \left\|\sum\limits_{u \in \P (\Z/N\Z)^{n-1}} \Ind_{L(u)}\right\|_{\ell^{n/(n-1)}}= \|h\|_{\ell^{n/(n-1)}}= \frac{\langle g,h\rangle}{\|g\|_{\ell^n}}.\label{eq-conj1}
\end{equation}

Let $f(x)=\lceil g(x) \rceil$. We note $g(x) \le f(x) \le 2g(x)$ and $\|f\|_\infty\le 2\|g\|_\infty \le 4N$. This also means that $\mweight(f,p_1)\le 4Np_1^{k}$.
By the definition of the maximal function we note $f^*(u)\ge \sum_{x\in L(u)} f(x)$. We let $f'(u)=\sum_{x\in L(u)} f(x)$ and $g'(u)=\sum_{x\in L(u)} g(u)$ for $u\in \P (\Z/N\Z)^{n-1}$.

Now, using $f^*\ge f'$ and Theorem~\ref{thm-maxN} we have,
\begin{equation}
    \|f\|_{\ell^n}\ge |\P (\Z/N\Z)^{n-1}|^{-1/n} C^{1/n} \|f^*\|_{\ell^n}\ge |\P (\Z/N\Z)^{n-1}|^{-1/n} C^{1/n} \|f'\|_{\ell^n} \label{eq-conj2}
\end{equation}

where

\begin{equation}
    C^{1/n}=\Omega\left(\frac{1}{\log_{p_1}(Np_1^kn)+1}\prod\limits_{i=2}^r\frac{1}{(k_i\log(p_i)+1)(2(k_i+\lceil \log_{p_i}(n)\rceil ))}\right)\ge  \Omega(N^{-o(1)}) \label{eq-conj3}
\end{equation}
using Fact~\ref{fact-Prest} ($o(1)$ is a function which goes to $0$ as $N\rightarrow \infty$ and $n$ is fixed). Putting \eqref{eq-conj1}, \eqref{eq-conj2}, and \eqref{eq-conj3} together with the fact that $g(x)\le f(x)\le 2g(x)$ we get,
\begin{align*}
\left\|\sum\limits_{u \in \P (\Z/N\Z)^{n-1}} \Ind_{L(u)}\right\|_{\ell^{n/(n-1)}}&\le 2 \frac{\langle h,f\rangle}{\|f\|_{\ell^n}} \le 2C^{-1/n} |\P (\Z/N\Z)^{n-1}|^{1/n} \frac{\langle h,f\rangle}{\|f'\|_{\ell^n}}\\
&\le O(N^{o(1)}) |\P (\Z/N\Z)^{n-1}|^{1/n} \frac{\langle h,f\rangle}{\|f'\|_{\ell^n}}\\
&=O(N^{o(1)}) |\P (\Z/N\Z)^{n-1}|^{1/n} \frac{\sum\limits_{x\in (\Z/N\Z)^n}\sum\limits_{u \in \P (\Z/N\Z)^{n-1}} \Ind_{L(u)}(x)f(x)}{\|f'\|_{\ell^n}}\\
&= O(N^{o(1)}) |\P (\Z/N\Z)^{n-1}|^{1/n} \frac{\sum\limits_{u \in \P (\Z/N\Z)^{n-1}} \sum\limits_{x\in L(u)} f(x)}{\|f'\|_{\ell^n}} \\
&= O(N^{o(1)}) |\P (\Z/N\Z)^{n-1}|^{1/n} \frac{\sum\limits_{u \in \P (\Z/N\Z)^{n-1}} f'(u)}{\|f'\|_{\ell^n}}.
\end{align*}
Finally, applying H\"older's inequality gives us,
\begin{align*}
\left\|\sum\limits_{u \in \P (\Z/N\Z)^{n-1}} \Ind_{L(u)}\right\|_{\ell^{n/(n-1)}}  &  \le O(N^{o(1)})|\P (\Z/N\Z)^{n-1}|^{1/n} \|1\|_{\ell^{n/(n-1)}(\P (\Z/N\Z)^{n-1})} \\
&\le O(N^{o(1)}) |\P (\Z/N\Z)^{n-1}|\le O(N^{n-1+o(1)}).
\end{align*}
As $N$ grows this completes the proof.
\end{proof}

\bibliographystyle{alpha}
\bibliography{bibliography}
\appendix

\section{Proof of Lemma~\ref{lem-goodmono}}
The proof here only involves a minor adaptation of the ideas in \cite{dhar2022linear}.

First, we state the multiplicity version of the  Schwartz-Zippel bound~\cite{schwartz1979probabilistic,ZippelPaper} 
(see \cite{DKSS13} for a proof). We denote by $\F[x_1,..,x_n]_{\le d}$ the space of polynomials of total degree at most $d$ with coefficients in $\F$.

\begin{lem}[Schwartz-Zippel with multiplicities]\label{multSchwartz}
Let $\F$ be a field, $d\in \Z_{\geq 0}$ and let $f\in \F[x_1,..,x_n]_{\le d}$ be a non-zero polynomial. Then, for any finite subset $U\subseteq \F$ ,
{\em $$\sum\limits_{a\in U^n} \textsf{mult}(f,a) \le d|U|^{n-1}.$$}
\end{lem}

\subsection{The EVAL matrix, its submatrices and their ranks}
Our main object of interest is the matrix encoding the evaluation of a subset of monomials (with their derivatives) on a subset of points.

\begin{define}[$\EVAL^m(S,W)$ matrix]
Let $\F$ be a field, and let $n,m \in \mathbb{N}$. Given a set $S \subset \F^n$ and a set of monomials $W\subset \F[x_1,\hdots,x_n]$, we let {\em $\EVAL^m(S,W)$} denote an  $|S|\binom{m-1+n}{n}\times |W|$ matrix whose columns are indexed by $W$ and rows are indexed by tuples $(x,\mathbf{j})\in S\times \Z_{\ge 0}^n$ s.t. $\text{wt}(\mathbf{j})<m$.
The $((x,\mathbf{j}),f)$th entry of this matrix is,
$$f^{(\mathbf{j})}(x).$$
In other words, the $(x,\mathbf{j})$th row of the matrix consists of the evaluation of the $\mathbf{j}$'th Hasse derivative of all $f\in W$ at $x$. Equivalently, the $f$'th column of the matrix consists of the evaluations of weight strictly less than $m$ Hasse derivatives of $f$ at all points in $S$.
\end{define}

Let $W_{d,n}$ denote the set of monomials in $n$-variables $x_1,\hdots,x_n$ of degree at exactly $d$. Our first lemma shows that the $\F_q$-rank of $\EVAL^m(\P \F_q^{n-1},W_{d,n})$ is maximal whenever $d$ is not too large. This is essentially the Schwartz-Zippel lemma since it means that a polynomial of bounded degree could be recovered from its evaluations (up to high enough order) on a product set and evaluations of homogenous polynomials over $\P\F_q^{n-1}$ gives us evaluations over the entirety of $\F_q^n$.

\begin{lem}[Rank of $\EVAL^m(\P \F_q^{n-1},W_{d,n})$]
Let $m\in \N$ then for all $d<mq$ we have,
{\em $$\text{rank}_{\F_q} \EVAL^m(\P \F_q^{n-1},W_{d,n})=|W_{d,n}|= \binom{d+n-1}{n-1}.$$}
\end{lem}
\begin{proof}
 Let $v_{m}$ be the column of $\EVAL^m(\P \F_q^{n-1},W_{d,n})$ corresponding to the monomial $m\in W_{d,n}$. The linear combination $\alpha_m v_m$ is precisely the evaluation of Hasse derivative of weight strictly less than $m$ of the homogenous polynomial $\alpha_m m$ over $\P \F_q^{n-1}$. If such a linear combination is zero then we get a homogenous linear polynomial of degree $d<mq$ which vanishes on the entirety of $\P \F_q^{n-1}$ with multiplicity at least $m$. As the polynomial is homogenous the same holds over the entirety of $\F_q^n$. This immediately contradicts Lemma~\ref{multSchwartz}.
\end{proof}

Our final lemma, which is the heart  of this section, shows that if we have a set $S$ of a $\delta$ fraction of points from $\P \F_q^{n-1}$ then $\text{rank}_{\F_q} \EVAL^m(S,W_{d,n})$ is at least a $\delta$ fraction of $\text{rank}_{\F_q} \EVAL^m(\P \F_q^{n-1},W_{d,n})$

\begin{lem}[Rank of $\EVAL^m(S,W_{d,n})$]
Let $m\in \N$ and $S\subseteq \P \F_q^{n-1}$ with $|S|\ge \delta |\P \F_q^{n-1}|,\delta\in [0,1]$ then for all $d<mq$ we have,
{\em $$\text{rank}_{\F_q} \EVAL^m(S,W_{d,n})\ge \delta  \cdot|W_{d,n}|=\delta \binom{d+n}{n-1}.$$}
\end{lem}
We note that Lemma~\ref{lem-goodmono} is a simple corollary of the above.
\begin{proof}
    Consider $S\subseteq \P \F_q^{n-1}$ such that $|S|=\delta |\P \F_q^{n-1}|$. For any $M\in \text{GL}_n(\F_q)$ let $M\cdot S=\{M\cdot y|y\in S\}$.
    
    \begin{claim}
    {\em $$\rank_{\F_q}\EVAL^m(S,W_{d,n})=\rank_{\F_q}\EVAL^m(M\cdot S,W_{d,n}).$$}
    \end{claim}
    \begin{proof}
        We will prove this statement by constructing an isomorphism between the column-space of the two matrices. An element in the column space of $\EVAL^m(S,W_{d,n})$ is the evaluation of the weight strictly less than $m$ Hasse derivatives on $S$ of a degree $d$ homogenous polynomial $f(x)\in \F_q[x_1,\hdots,x_n]$. We map such a vector to the evaluation of the weight strictly less than $m$ Hasse derivatives on $M\cdot S$ of the polynomial $f(M^{-1}x)$ which will also be homogenous of degree $d$. The choice of $f$ in the beginning can be ambiguous but if there are two polynomials $f(x)$ and $g(x)$ having the same evaluation of weight strictly less than $m$ Hasse derivatives over $S$ then $f(x)-g(x)$ vanishes on $S$ with multiplicity at least $m$. By Lemma~\ref{lem:chainRule}, $f(M^{-1}x)-g(M^{-1}x)$ vanishes on $M\cdot S$ with multiplicity $m$ which implies $f(M^{-1}x)$ and $g(M^{-1}x)$ evaluate to the same weight strictly less than $m$ Hasse derivatives over $M\cdot S$.
        The inverse map can be similarly constructed.
    \end{proof}
    
    The above claim shows it suffices to show the rank bound for any $M\cdot S$ where $M\in \text{GL}_n(\F_q)$. We do this by a probabilistic method argument.
    
    The previous Lemma implies that $\EVAL^m(\P \F_q^{n-1},W_{d,n})$ has $\F_q$-rank $|W_{d,n}|$. As this is a matrix with $\F_q$ entries this means that there exists a $|W_{d,n}|=\binom{d+n-1}{n-1}$ subset of rows $R$ of $\EVAL^m(\P \F_q^{n-1},W_{d,n})$ which are linearly independent.
    
    These rows are indexed by tuples 
    $$(x,\bi)\in \P \F_q^{n-1}\times \Z_{\ge 0}^n$$
    with $\text{wt}(\bi)<m$. The $(x,\bi)$th row is the evaluation of the $\bi$th Hasse Derivative at $x$ of the monomials in $W_{d,n}$.
    
    We pick an $M\in \text{GL}_n(\F_q)$ uniformly at random. We now calculate the expected fraction of the rows from $R$ which appear in $\EVAL^m(M\cdot S, W_{d,n})$. 
    
    A row in $R$ indexed by $(x,\bi)\in \P \F_q^{n-1}\times \Z_{\ge 0}^n$ will appear in $\EVAL^m(M\cdot S, W_{d,n})$ if and only if $x\in M\cdot S$. As the action of $\text{GL}_n(\F_q)$ on $\P\F_q^{n-1}$ is transitive we see that this happens with probability at least $\delta$. This means the expected fraction of rows in $R$ appearing in $\EVAL^m(M\cdot S, W_{d,n})$ is at least $\delta$. This ensures that there is some matrix $M$ such that $\EVAL^m(M\cdot S, W_{d,n})$ and hence $\EVAL^m(M\cdot S, W_{d,n})$ has $\F_q$-rank at least $\delta|W_{d,n}|$. 
\end{proof}

\section{Proof of Lemma \ref{lem:1rank}}\label{ap:rk}

As we need a small modification, we reproduce the analysis of \cite{DharGeneral, arsovski2021padic} here. We first show $M_{\ell,1}$ has an explicit decomposition as a product of lower and upper triangular matrices.

\begin{lem}[Lemma 5 in \cite{arsovski2021padic}]\label{lem:diag}
Let $V_m$ for $m\in \N$ be an $m\times m$ matrix whose row and columns are labelled by elements in $\{0,\hdots,m-1\}$ such that its $i,j$th entry is $z^{ij}\in \Z[z]$. 

In this setting, there exists a lower triangular matrix $L_{m}$ over $\Z[z]$ with ones on the diagonal such that its inverse is also lower triangular with entries in $\Z[z]$ with ones on the diagonal, and an upper triangular matrix $D_{m}$ over $\Z[z]$ whose rows and columns are indexed by points in $\{0,\hdots,m-1\}$ such that the $j$th diagonal entry for $j\in \{0,\hdots,m-1\}$ equals

 \begin{align*}
 D_{m}(j,j)=      \prod_{i=0}^{j-1} (z^j-z^i)
 \end{align*}
 such that $V_{m}=L_{m}D_{m}$.
\end{lem}
This statement is precisely Lemma 5 in \cite{arsovski2021padic}. It is easy to prove by doing elementary row operations on the Vandermonde matrix $V_{m}$.

We will also need Lucas's theorem from \cite{Lucas}.
\begin{thm}[Lucas's Theorem~\cite{Lucas}]\label{thm:luc}
Let $p$ be a prime and Given two natural numbers $a$ and $b$ with expansion $a_k p^k+\hdots+a_1p+a_0$ and  
$b_kp^k+\hdots+b_0$ in base $p$ we have,
\em{$$\binom{a}{b} \text{(mod }p\text{)}=\prod\limits_{i=0}^k\binom{a_i}{b_i} \text{(mod }p\text{)} .$$}
A particular consequence is that $\binom{a}{b}$ is non-zero if and only if every digit in base-$p$ of $b$ is at most as large as every digit in base-$p$ of $a$.
\end{thm}

\begin{replem}{lem:1rank}
$m\ge \ell$ then the $\F_p$-rank of $M^\ell_{\ell,n}$ is at least {\em $$\text{rank}_{\F_p} M_{\ell,n} \ge \left\lceil\binom{\ell \lceil\log_p(\ell)\rceil^{-1} +n}{n}\right\rceil.$$}
\end{replem}

\begin{proof}
Using the previous lemma we note that $V_{m}=L_mD_m$. Under the ring map $f$ from $\Z[z]$ to $\Z[z]/\langle (z-1)^\ell\rangle$, $V_{m}=L_{m}D_{m}$ becomes $M^\ell_{m,1}=\overline{L}_{m}\overline{D}_{m}$ where $\overline{L}_{m}$ and $\overline{D}_{m}$ are the matrices $f(L_{\ell})$ and $f(D_{\ell})$ respectively.

We next notice that $M^\ell_{m,n}$ is $M_{\ell,1}$ tensored with itself $n$ times which we denote as $M^\ell_{m,n}=(M^\ell_{m,1})^{\otimes n}$. Using $M^\ell_{m,1}=\overline{L}_{m}\overline{D}_{m}$ and Fact \ref{fact:multOfKronecker} we have $M^\ell_{m,n}=(M^\ell_{m,1})^{\otimes n}=\overline{L}_{m}^{\otimes n}\overline{D}_{m}^{\otimes n}$. As $L_m$ was invertible with its inverse also having entries in $\Z[z]$ we see that $\overline{L}_m$ is also invertible and $\left(\overline{L}_{m}^{\otimes n}\right)^{-1}M^\ell_{m,n}= \overline{D}_{m}^{\otimes n}$. By standard properties of the tensor product have that,
$$\rank_{\F_p} M^\ell_{m,n}\ge \rank_{\F_p} \overline{D}_{m}^{\otimes n}.$$
As $\overline{D}_{m}$ is upper triangular so will $\overline{D}_{m}^{\otimes n}$ be. Therefore, to lower bound the rank of $\overline{D}_{m}^{\otimes n}$ we can lower bound the number of non-zero diagonal elements. 

The diagonal elements in $\overline{D}_{m}^{\otimes n}$ correspond to the product of diagonal elements chosen from $n$ copies of $\overline{D}_{m}$. Recall, the rows and columns of $\overline{D}_{m}$ are labelled by $j\in \{0,\hdots,m-1\}$ with
\begin{align*}
 D_{m}(j,j)=      \prod_{i=0}^{j-1} (z^j-z^i)
 \end{align*}
Setting $z-1=w$ we note that $\F_p[z]/\langle (z-1)^\ell\rangle$ is isomorphic to $\F_p[w]/\langle w^{\ell}\rangle$. $D_{m}(j,j)$ can now be written as
\begin{align*}
 D_{m}(j,j)=(1+w)^{j(j-1)/2} \prod_{i=1}^{j} ((1+w)^i-1)
 \end{align*}
Using Lucas's Theorem (Theorem \ref{thm:luc}) we see that the largest power of $t$ which divides $(1+w)^l-1$ is the same as the largest power of $p$ which divides $l$. For any $j\le \ell-1$ , therefore the largest power of $w$ which divides $\overline{D}_{m}(j,j)$ is at most
\begin{align*}
    \sum\limits_{t=0}^{\lfloor \log_p(j)\rfloor} \left(\left\lfloor \frac{j}{p^t} \right\rfloor-\left\lfloor \frac{j}{p^{t+1}} \right\rfloor\right)p^t&=j+\sum\limits_{t=1}^{\lfloor \log_p(j)\rfloor} \left\lfloor \frac{j}{p^t} \right\rfloor p^{t-1}(p-1)\\
    &\le j(1+\lfloor \log_p(j)\rfloor(1-1/p))\\
    &\le j(\lfloor\log_p(\ell-1)\rfloor+1-\lfloor\log_p(\ell-1)\rfloor/p).\numberthis \label{eq:jBd}
\end{align*}
Consider the set of tuples $(j_1,\hdots, j_n)\in \N$ such that $j_1+\hdots+j_n\le \lceil \ell/(\lfloor\log_p(\ell-1)\rfloor+1) \rceil $. Using \eqref{eq:jBd} we see that the diagonal entry in $\overline{D}_{m}^{\otimes n}$ corresponding to the tuple will be divisible by at most $w^{\lceil \ell/(\lfloor\log_p(\ell-1)\rfloor+1) \rceil (\lfloor\log_p(\ell-1)\rfloor+1-\lfloor\log_p(\ell-1)\rfloor/p)}$. It is easy to check that the exponent is at most $\ell-1$ which will guarantee that the $(j_1,\hdots,j_n)$'th diagonal entry of $\overline{D}_{m}^{\otimes n}$ is non-zero. 

We note as $\lfloor\log_p(\ell-1)\rfloor+1=\lceil\log_p(\ell)\rceil$ we get at least $\lceil\binom{\ell/\lceil \log_p(\ell)\rceil^{-1} +n}{n}\rceil$ non-zero diagonal entries proving the desired rank bound.
\end{proof}

\section{Proof of Lemma~\ref{lem:deRich}}

 We need two simple facts for the proof.

\begin{fact}[Hasse Derivatives of composition of two functions]\label{fact:com}
Let $\F$ be a field, $n\in \N$. Given a tuple of polynomials $C(y)=(C_1(y),C_2(y),\hdots,C_n(y))\in (\F[y])^n$, $w\in \N$ and $\gamma\in \F$ there exists a set of coefficients $b_{w,\balpha}\in \F$ (which depend on $C$ and $\gamma$) where $\balpha \in \Z^n_{\ge 0}$ such that for any $f\in \F[x_1,\hdots,x_n]$ we have,
$$h^{(w)}(\gamma)=\sum\limits_{\text{wt}(\balpha)\le w} b_{w,\balpha} f^{(\balpha)}(C_1(\gamma),\hdots,C_n(\gamma)), $$
where $h(y)=f(C_1(y),\hdots,C_n(y))$. 

\end{fact}

This fact follows easily from the definition of the Hasse derivative. A proof can also be found in Proposition 6 of \cite{DKSS13}. We also need another fact about the isomorphism between polynomials and the evaluations of their derivatives at a sufficiently large set of points.

\begin{fact}[Computing polynomial coefficients from polynomial evaluations]\label{fact:eval}
Let $\F$ be a field and $n\in \N$. Given distinct $a_i\in \F$ and $\beta_i\in \Z_{\ge 0},i=1,\hdots,n$, let $h(y)=\prod_{i=1}^n (y-a_i)^{\beta_i}\in \F[y]$. We have an isomorphism between
$$\frac{\F[z]}{\left\langle h(z)\right\rangle } \longleftrightarrow \F^{\sum\limits_{i=1}^n \beta_i},$$
which maps every polynomial $f\in \F[z]/\langle h(z)\rangle$ to the evaluations $(f^{(j_i)}(a_i))_{i,j_i}$ where $i\in \{1,\hdots,n\}$ and $j_i\in \{0,\hdots,\beta_i-1\}$. 
\end{fact}

This is a simple generalization of the fact that for a univariate polynomial having its first $\beta$ Hasse derivatives at $a$ be $0$ is equivalent to it being divisible by $(z-a)^\beta$. It can be proven in several ways, one way of proving it would be using the previous statement with the Chinese remainder theorem for the ring $\F[y]$. To prove Lemma~\ref{lem:deRich} we will need the following corollary of the fact above.

\begin{cor}[Computing a polynomial from its evaluations]\label{cor:eval}
Let $\F$ be a field and $n\in \N$. Given distinct $a_i\in \F$ and $\beta_i\in \Z_{\ge 0}$, let $h(y)=\prod_{i=1}^n (y-a_i)^{\beta_i}\in \F[y]$. Then there exists coefficients $t_{i,j}\in \F[z]$ (depending on $h$) for $i\in \{1,\hdots,n\},j\in \{0,\hdots,\beta_i-1\}$ such that for any $f(y)\in \F[y]$ we have,
$$\sum\limits_{i=1}^n \sum\limits_{j=0}^{\beta_i-1} t_{i,j} f^{(j)}(a_i)=f(z) \in  \F[z]/\langle h(z)\rangle.$$
\end{cor}
\begin{proof}
Fact \ref{fact:eval} implies that there exists a $\F$-linear map which can compute the coefficients of $1,z,z^2,\hdots, z^{\sum_{i=1}^n \beta_i-1}$ of $f(z) \in \F[z]/\langle h(z)\rangle$ from the evaluations $f^{(j_i)}(a_i)$ for $i\in \{1,\hdots,n\}$ and $j_i\in~\{0,\hdots,\beta_i-1\}$. Multiplying these coefficients with $1,\hdots, z^{\sum_{i=1}^n \beta_i-1}$ computes $f(z)$ in $\F[z]/\langle h(z)\rangle$.
\end{proof}

We are now ready to prove Lemma~\ref{lem:deRich}.

\begin{proof}[Proof of Lemma~\ref{lem:deRich}]
As the statement we are trying to prove is linear over $\Z$ we see that it suffices to prove the lemma in the case of when $f$ equals a monomial. Given $v\in \Z_{\ge 0}^n$ we let $Q_v(x)=Q_v(x_1,\hdots,x_n)=x_1^{v_1}\hdots x_n^{v_n}$ be a monomial in $\F[x_1,\hdots,x_n]$ where $\F$ is an arbitrary field (we will be working with $\F=\Q(\zeta)$ and $\F=\F_p$).
Let $C(y)=y^{u'}=(y^{u'_1},y^{u'_2},\hdots,y^{u'_n})\in (\F[y])^n$ where $u'=(u'_1,\hdots,u'_n)\in \Z^n,u=(u_1,\hdots,u_n)\in \P (\Z/p^k\Z)^{n-1}$ and $u' \text{ (mod }p^k\text{)}=u$. For this proof we use the elements in $ \{0,\hdots, p^k -1\}$ to represent the set $\Z/p^k\Z$.

We first prove the following claim.
\begin{claim}\label{cl:inter}
Let $w\in \N, \lambda \in \Z/p^k\Z$. There exists coefficients $b'_{w,\balpha}(\lambda)\in \Q(\zeta)$ (depending on $w,\lambda$ and $C$) for $\balpha\in \Z_{\ge 0}^n$ with $\text{wt}(\balpha)\le~w$ such that for all monomials $Q_v(x)\in \Z[x_1,\hdots,x_n],v\in\Z_{\ge 0}^n$ we have,
$$\sum\limits_{\text{wt}(\balpha)\le w} b'_{w,\balpha}(\lambda) Q_v^{(\balpha)}(\zeta^{a+\lambda u})=\zeta^{\langle a,v\rangle} (Q_v\circ C)^{(w)}(\zeta^\lambda).$$
\end{claim}
\begin{proof}
For every $\lambda \in \Z/p^k\Z$ and $w\in \N$, using Fact \ref{fact:com} we can find coefficients $b_{w,\balpha}(\lambda)\in \Q(\zeta)$ such that,
\begin{align}\label{eq:Der1}
(G\circ C)^{(w)}(\zeta^{\lambda})= \sum\limits_{\text{wt}(\balpha)\le w} b_{w,\balpha}(\lambda) G^{(\balpha)}(C(\zeta^{\lambda}))= \sum\limits_{\text{wt}(\balpha)\le w} b_{w,\balpha}(\lambda) G^{(\balpha)}(\zeta^{\lambda u'}),
\end{align}
for every polynomial $G\in \Q(\zeta)[x_1,\hdots,x_n]$ where $(G\circ C)(y)=f(y^{u'})\in \Q(\zeta)[y]$. As $u' \text{ (mod }p^k\text{)}=u$ and $\zeta$ is a primitive $p^k$'th root of unity in $\C$ we note that 
$$\zeta^{a+\lambda u}=\zeta^{a+\lambda u'}.$$
We now make the simple observation that for any $\balpha \in \Z_{\ge 0}^n$ and $v\in \Z_{\ge 0}^n$ we have 
$$Q^{(\balpha)}_v(x)=\prod\limits_{i=1}^n\binom{v_i}{\alpha_i} x_i^{v_i-\alpha_i},$$
which implies
$$Q^{(\balpha)}_v(\zeta^{a+\lambda u})=\zeta^{\langle a,v\rangle}\zeta^{-\langle \balpha,a\rangle}Q^{(\balpha)}_v(\zeta^{\lambda u})=\zeta^{\langle a,v\rangle}\zeta^{-\langle \balpha,a\rangle}Q^{(\balpha)}_v(\zeta^{\lambda u'}).$$
The above equation combined with \eqref{eq:Der1} for $G=Q_v$ implies,
\begin{align*}
    \zeta^{\langle a,v\rangle} (Q_v\circ C)^{(w)}(\zeta^\lambda)=\sum\limits_{\text{wt}(\balpha)\le w} \zeta^{\langle \balpha,a\rangle}b_{w,\balpha}(\lambda) Q_v^{(\balpha)}(\zeta^{a+\lambda u}),
\end{align*}
for all $w\in \N$ and $v\in \Z_{\ge 0}^n$. Setting $b'_{w,\balpha}(\lambda)=\zeta^{\langle \balpha,a\rangle}b_{w,\balpha}(\lambda)$ we are done .
\end{proof}

Without loss of generality let us assume $\sum_{x\in L} \pi(x)=\ell$ (if it is greater we can reduce each of the $\pi(x)$ until we reach equality - this would just mean that our computation was done by ignoring some higher order derivatives at some of the points).

Let $h(y)\in \Z(\zeta)[y]\subseteq \Q(\zeta)[y]$ be the polynomial,
$$h(y)= \prod\limits_{\lambda \in \Z/p^k\Z}(y-\zeta^{\lambda})^{\pi(a+\lambda u)}.$$

Using Corollary~\ref{cor:eval} there is a $\Q(\zeta)[z]$-linear combination of the evaluations  $(Q_v\circ C)^{(w)}(\zeta^\lambda)$ for $\lambda \in \Z/p^k\Z$ and $w<\pi(a+\lambda u)$ which can compute the element 
$$(Q_v\circ C)(z)=z^{\langle v,u'\rangle} \in \Q(\zeta)[z]/\langle h(z)\rangle.$$ 
This statement along with Claim~\ref{cl:inter} leads to the following: there exists elements $c_{\lambda,\balpha} \in \Q(\zeta)[z]$ (depending on $\pi,L$ and $u'$) for $\lambda \in \Z/p^k\Z$ and $\balpha\in  \Z_{\ge 0}^n$ with $\text{wt}(\balpha)< \pi(a+\lambda u)$ such that the following holds for all monomials $m_v\in \Z[x_1,\hdots,x_n],v\in \Z_{\ge 0}$ we have,

$$\sum\limits_{\lambda =0}^{p^k-1}\sum\limits_{\text{wt}(\balpha) < \pi(a+\lambda u) }c_{\lambda,\balpha} Q_v^{(\balpha)}(\zeta^{a+\lambda u})=\zeta^{\langle a,v\rangle} (m_v\circ C)(z)=\zeta^{\langle a,v\rangle}z^{\langle v,u'\rangle} \in \Q(\zeta)/\langle h(z)\rangle.$$

We claim that these are coefficients we wanted to construct in the statement of this lemma. 

All we need to show now is that $\psi_{p^k}$ is a ring homomorphism from the ring $\Z(\zeta)/\langle h(z)\rangle$ to the ring $\overline{T}_\ell=\F_p(\zeta)/\langle (z -1)^\ell\rangle$ and maps $\zeta^{\langle a,v\rangle}z^{\langle v,u'\rangle}$ to $z^{\langle v,u'\rangle}$. This follows from Corollary~\ref{cor:psiE} and noting 
$$\psi_{p^k}(h(z))=(z-1)^{\sum\limits_{\lambda \in\Z_{p^k}} \pi(a+\lambda u)}=(z-1)^{\ell}\in \F_p[z].$$

\end{proof}
\end{document}